\documentclass[a4paper,reqno, 12pt]{amsart}

\usepackage[margin=3cm]{geometry}

\makeatletter
\let\old@setaddresses\@setaddresses
\def\@setaddresses{\bigskip\bgroup\parindent 0pt\let\scshape\relax\old@setaddresses\egroup}
\makeatother

\usepackage[utf8]{inputenc}
\usepackage[T1]{fontenc}

\usepackage{amsthm, amssymb, amsmath}

\usepackage{thmtools}
\usepackage{thm-restate}

\newtheorem{theorem}{Theorem}[section]
\newtheorem{lemma}[theorem]{Lemma}

\newtheorem{corollary}[theorem]{Corollary}
\newtheorem{observation}[theorem]{Observation}
\newtheorem{problem}[theorem]{Problem}

\usepackage{float}
\usepackage{needspace}
\usepackage{graphicx}
\usepackage{xcolor}
\usepackage{subcaption}

\usepackage{tikz} 
\usetikzlibrary{shapes,arrows,animations,quotes}
\usetikzlibrary {graphs}
\usetikzlibrary{decorations.pathreplacing}
\usetikzlibrary {arrows.meta}

\newcommand{\ve}{\textup{\textsf{v}}}
\newcommand{\e}{\textup{\textsf{e}}}

\usepackage{hyperref}
\hypersetup{
  pdftitle = {Nowhere-zero flow reconfiguration},
  pdfauthor=  {Louis Esperet, Aurélie Lagoutte, Margaux Marseloo},
  colorlinks = true,
  linkcolor = black!30!red,
  citecolor = black!30!green
}
\usepackage[shortlabels]{enumitem}

\usepackage{marvosym}
\newcommand{\apx}[1]{\hyperref[#1]{\Rightscissors}}

\def\cqedsymbol{\ifmmode$\lrcorner$\else{\unskip\nobreak\hfil
\penalty50\hskip1em\null\nobreak\hfil$\lrcorner$
\parfillskip=0pt\finalhyphendemerits=0\endgraf}\fi}

\let\le\leqslant
\let\ge\geqslant
\let\leq\leqslant
\let\geq\geqslant

\DeclareMathOperator{\supp}{\mathrm{supp}}

\title{Nowhere-zero flow reconfiguration}

\author[L.~Esperet]{Louis Esperet}
\address[L.~Esperet]{Laboratoire G-SCOP (CNRS, Université Grenoble Alpes), Grenoble, France}
\email{louis.esperet@grenoble-inp.fr}

\author[K.~Hendrey]{Kevin Hendrey}
\address[K.~Hendrey]{School of Mathematics, Monash University, Melbourne, Australia}
\email{kevin.hendrey1@monash.edu}

\author[A.~Lagoutte]{Aurélie Lagoutte}
\address[A.~Lagoutte]{Laboratoire G-SCOP (CNRS, Université Grenoble Alpes), Grenoble, France}
\email{aurelie.lagoutte@grenoble-inp.fr}

\author[M.~Marseloo]{Margaux Marseloo}
\address[M.~Marseloo]{Université de Genève, Switzerland}
\email{margaux.marseloo@unige.ch}

\author[S.~Norin]{Sergey Norin}
\address[S.~Norin]{McGill University, Montréal, Canada}
\email{snorin@math.mcgill.ca}

\author[R.~Steiner]{Raphael Steiner}
\address[R.~Steiner]{Department of Mathematics, Institute for Operations Research, ETH Z\"urich, Switzerland}
\email{raphaelmario.steiner@math.ethz.ch}

\thanks{L.~Esperet is partially supported by the French ANR Projects
  ENEDISC (ANR-24-CE48-7768) and Mimétique (ANR-25-CE48-4089), and by LabEx
  PERSYVAL-lab (ANR-11-LABX-0025). A.~Lagoutte is partially supported by the French ANR project GRALMECO (ANR-21-CE48-0004). S.~Norin is partially supported by an NSERC discovery grant.}

\begin{document}

\begin{abstract}
We initiate the study of nowhere-zero flow reconfiguration. The natural question is whether any two nowhere-zero $k$-flows of a given graph $G$ are connected by a sequence of nowhere-zero $k$-flows of $G$, such that any two consecutive flows in the sequence differ only on a cycle of $G$. 

We study this problem in the setting of integer flows and group flows, and prove a number of positive and negative results. 
\begin{itemize}
    \item The natural  reconfiguration variant of Tutte's 5-flow conjecture, stating that any two nowhere-zero 5-flows in any 2-edge-connected graph are connected, is false in the group and integer cases. 
    \item All nowhere-zero $\mathbb{Z}_2^8$-flows of every 2-edge-connected graph are connected and for every sufficiently large abelian group $A$, all nowhere-zero $A$-flows of every 2-edge-connected graph are connected.
\item The group structure  affects the answer, contrary to the existence problem for nowhere-zero flows.
\item We highlight a duality with recoloring in planar graphs and deduce that any two nowhere-zero 7-flows in a planar graph are connected, among other results.
\item For every 2-edge-connected graph $G$, there is an integer $k$ such that all nowhere-zero $k$-flows of $G$ are connected. 
\end{itemize}
\end{abstract}

\maketitle

\section{Introduction}

Combinatorial reconfiguration studies the space of solutions of a
given combinatorial problem through the lens of its reconfiguration
graph, where two solutions are adjacent if one can be obtained from
the other by an elementary operation.
The main
questions considered in combinatorial reconfiguration are the
following: (1) is the reconfiguration graph  connected? and if so, (2)
what is the diameter? Results on the connectivity and
expansion of the reconfiguration graph have immediate consequences on
the efficiency of random sampling algorithms based on Markov chains,
and are of interest for various communities ranging from algorithmic
graph theory to statistical physics.

\smallskip

 In this paper we initiate the study of  nowhere-zero flow
 reconfiguration. The vertices of the reconfiguration graph are
 nowhere-zero flows, and an elementary operation consists in adding some flow value along a cycle.

\smallskip

We study two natural settings: integer flows and group flows. Given an
(additive) abelian group $A$, an $A$-flow in a graph $G$ is an assignment of
values and directions to the edges of $G$ such that around each vertex
$v$, the sum of the incoming flow values is equal to the sum of the
outgoing flow values. For an integer  $k\ge 2$, a 
\emph{$k$-flow} in a graph $G$ is a $\mathbb{Z}$-flow in which all flow
values are less than $k$ in absolute value. 
An $A$-flow is said to be \emph{nowhere-zero} if no edge is assigned $0_A$,
the neutral element of $A$. We consider the reconfiguration graphs
$\mathcal{F}(G,A)$ and $\mathcal{F}(G,k)$ whose vertices are the
nowhere-zero $A$-flows and $k$-flows of $G$ (respectively), and in which two
flows are adjacent if and only the support of their difference is a cycle
(i.e., a
2-regular connected subgraph). Our main concern is thus to identify for which graph $G$ and group $A$ (resp.\ integer $k$) the graph $\mathcal{F}(G,A)$ (resp.\ $\mathcal{F}(G,k)$) is connected (note that the empty graph is considered as connected in this paper).

\smallskip

A classical result of Tutte  \cite{Tutte190} states that nowhere-zero integer and group flows
are equivalent, in the sense that for any abelian group $A$, a graph $G$ has a nowhere-zero
$A$-flow if and only if it has a nowhere-zero
$|A|$-flow. In particular, for any two abelian groups $A$ and $B$ of
the same cardinality, a graph has a nowhere-zero $A$-flow if and only
if it has a nowhere-zero $B$-flow, in other words the existence of
nowhere-zero $A$-flows does not depend on the structure of an abelian group $A$, only
on its cardinality.

\smallskip

We first observe that if $\mathcal{F}(G,k)$ is connected, then
$\mathcal{F}(G,\mathbb{Z}_k)$ is also connected, where $\mathbb{Z}_k$ denotes the cyclic group of all
integers modulo $k$. It is not yet clear whether the converse also holds, but we provide examples of graphs
for which  $\mathcal{F}(G,\mathbb{Z}_2\times \mathbb{Z}_2)$ is
connected while  $\mathcal{F}(G,\mathbb{Z}_4)$ is
not. This shows that the connectivity of the flow reconfiguration
graph does depend on the group structure. As most of the existential
results for integer flows are based on groups flows for various
groups, this suggests that integer flow reconfiguration is
significantly harder than group flow reconfiguration. 

\smallskip

A classical conjecture of Tutte \cite{Tutte190} states that every 2-edge-connected graph
has a nowhere-zero 5-flow, and the best known result towards this
conjecture is a theorem of Seymour \cite{Seymour6flows}, stating that every 2-edge-connected graph
has a nowhere-zero $(\mathbb{Z}_2\times \mathbb{Z}_3)$-flow, and thus
a nowhere-zero 6-flow. A natural problem is the following
reconfiguration version of the existence of nowhere-zero 5-flows:

\begin{problem}\label{conj:iflow}
Is it true that for every 2-edge-connected graph $G$, 
the reconfiguration graph $\mathcal{F}(G,5)$ is connected?
\end{problem}

Note that Problem \ref{conj:iflow} is incomparable to Tutte's 5-flow conjecture, since the connectedness of $\mathcal{F}(G,k)$ does not imply its non-emptiness. 
The following weaker problem is also of interest.

\begin{problem}\label{conj:iflow2}
Is there an integer $k\ge 5$, such that for every 2-edge-connected graph $G$,
the reconfiguration graph $\mathcal{F}(G,k)$ is connected?
\end{problem}

As explained above, it is meaningful to study the (possibly weaker)
variants of these problems on group flows.

\begin{problem}\label{conj:gflow1}
Is it true that for every 2-edge-connected graph $G$, 
the reconfiguration graph $\mathcal{F}(G,\mathbb{Z}_5)$  is connected?
\end{problem}

\begin{problem}\label{conj:gflow2}
Is there an abelian group $A$, such that for
every 2-edge-connected graph $G$, the reconfiguration graph
$\mathcal{F}(G,A)$ is connected?
\end{problem}

One of the main
techniques in combinatorial reconfiguration in order to show that a
reconfiguration graph is disconnected is to identify isolated
vertices, also called \emph{frozen configurations}. This is how many
lower bounds in graph recoloring are proved, for instance. We show that for
any graph $G$, these
frozen configurations do not exist in the flow reconfiguration graph
$\mathcal{F}(G,A)$, whenever $|A|=4$ or
$|A|\ge 6$, and in the flow reconfiguration graph
$\mathcal{F}(G,k)$ when $k\ge 2$. On the other hand, we show that surprisingly, $\mathcal{F}(G,\mathbb{Z}_5)$ can contain isolated vertices. This gives in particular negative answers to Problems \ref{conj:iflow} and \ref{conj:gflow1} above.

\smallskip

We prove that for any 2-edge-connected planar graph $G$ and any
integer $k\ge 7$, and any group $A$ of cardinality at least 7,
$\mathcal{F}(G,k)$ and $\mathcal{F}(G,A)$ are both
connected. We do so by showing that in planar graphs, the flow
reconfiguration problem is dual to the graph recoloring problem, where
elementary operations consist of recoloring vertices one at a
time. Moreover, upper bounds on the diameter of the recoloring
graph can be transferred to the flow reconfiguration graph.

\smallskip

We also make a connection between the reconfiguration of nowhere-zero
$(\mathbb{Z}_2\times \mathbb{Z}_2)$-flows in cubic graphs and the
reconfiguration of 3-edge-colorings, where the elementary operations
are Kempe changes. We can then translate known results in this setting to flow
reconfiguration, in particular we deduce that for every
2-edge-connected planar cubic bipartite graph $G$, the reconfiguration graph
$\mathcal{F}(G,\mathbb{Z}_2\times \mathbb{Z}_2)$ is
connected.

\smallskip

Our main result is that for any 2-edge-connected graph $G$, $\mathcal{F}(G,\mathbb{Z}_2^8)$ is connected, giving a positive answer to Problem \ref{conj:gflow2} above. We then show that for any sufficiently large abelian group $A$ and for any 2-edge-connected graph $G$, $\mathcal{F}(G,A)$ is connected. This extension is not immediate, contrary to the existence problem for group flows (see the remark after Theorem \ref{thm:tuttecard}). 

Finally, we show that for any 2-edge-connected graph $G$,
there exists an integer $k$ for which $\mathcal{F}(G,k)$ is
connected, a small step towards a resolution of Problem \ref{conj:iflow2}.

\subsection*{Outline} We start with preliminaries on flows and reconfiguration in Section \ref{sec:prel}. We then prove in Section \ref{sec:frozen} that no frozen configuration can exist, except for groups $A$ with $|A|\le 3$ or $A=\mathbb{Z}_5$. We then precisely characterize frozen nowhere-zero $\mathbb{Z}_5$-flows and use this characterization to find such a flow explicitly, giving a negative answer to Problems \ref{conj:iflow} and \ref{conj:gflow1}.
Section \ref{sec:z2z2} is dedicated to the study
of $(\mathbb{Z}_2\times \mathbb{Z}_2)$-flow reconfiguration, and its connection to Kempe changes in 3-edge-colorings of cubic graphs. In Section \ref{sec:duality}, we introduce the classical duality between flows and colorings in planar graphs and explain how it translates to the setting of reconfiguration. In Section \ref{sec:ab}, we prove that for any 2-edge-connected graph $G$, $\mathcal{F}(G,\mathbb{Z}_2^8)$ is
connected, giving a positive answer to Problem \ref{conj:gflow2}. We then extend this result to all sufficiently large abelian groups. In Section \ref{sec:intflows}, we show that for every 2-edge-connected graph $G$, there is an integer $k$ such that $\mathcal{F}(G,k)$ is connected. We conclude in Section \ref{sec:ccl} with some additional remarks.

\section{Preliminaries}\label{sec:prel}

\subsection{Group and integer flows}\label{sec:def}

All abelian groups in this paper are written additively,
and the neutral element of an abelian group $A$ is denoted by $0_A$,
or simply $0$ if $A$ is clear from the context. All graphs considered in the paper
are allowed to have multiple edges, but no loop.

\smallskip

    Let $G$ be a 2-edge-connected graph, let $D$ be an orientation of its edges and let $A$ be an abelian group. An \emph{$A$-flow} in $G$ with respect to $D$ is a function $f : E \to A$ such that, for every vertex $v \in V$, the following conservation condition holds: \[
        \sum_{e\in E^+(v)} f(e) = \sum_{e\in E^-(v)} f(e),
    \] where $E^+(v)$, respectively $E^-(v)$, is the set of all arcs
    of $D$ with their tails, respectively their head, at $v$.
    In the following, we denote $\sum_{e\in E^*(v)} f(e)$ by $f^*(v)$
    with $* \in \{+, -\}$ and $\delta f (v) = f^+(v) - f^-(v)$.

    \smallskip
    
    A $\mathbb{Z}$-flow $f$ in a graph $G$ is said to be a
    \emph{$k$-flow} if  $|f(e)| < k$ for every edge $e \in E(G)$.
    We note that given a graph $G$, an orientation $D$ of
$E(G)$, and an abelian group $A$, the $A$-flows in $G$ with respect
to $D$ naturally form an abelian group, where $(f+g)(e)=f(e)+g(e)$,
and in
particular we can define the difference $f-g$ of two $A$-flows. Note
that when
$f$ and $g$ are $k$-flows, $f+g$ and $f-g$ are still
$\mathbb{Z}$-flows, but not necessarily $k$-flows.

\smallskip

  The \emph{support} of an $A$-flow $f : E \to A$, denoted $\supp(f)$,
  is the set of all edges of $G$ with $f(e)\neq 0_A$. A flow $f$ on a graph $G$ is said to be \emph{nowhere-zero} if $\supp(f) = E(G)$.
Note that the existence of a nowhere-zero $A$-flow in a graph $G$ is
independent of the underlying orientation $D$ of the edge-set of
$G$ (since any arc $uv$ with flow value $p$ can be reversed to the arc $vu$ with flow value $-p$). So in the following, we often omit to mention $D$ explicitly and
instead consider it
to be fixed implicitly whenever we consider a graph $G$. Crucially,
all the flows of a graph are defined using the same underlying orientation.

\medskip

For a graph $G=(V,E)$ we write $\ve(G)=|V|$ and $\e(G)=|E|$. We will need the following classical results about flow decomposition into cycles.

\begin{lemma}\label{lem:linear}
    For any connected graph $G$ and abelian group $A$, any $A$-flow in $G$ is the sum of at most $\e(G)-\ve(G)+1$ $A$-flows whose supports
    are cycles of $G$.
  \end{lemma}
  
  \begin{proof}
   Let $f$ be an $A$-flow in $G$.   As long as $\supp(f)$ contains a cycle $C$, subtract from $f$ a flow along $C$ whose value is equal to $f(e)$, for some $e\in C$. Note that the support of $f$ decreases by at least 1 at each step. When $\supp(f)$ is a forest, $\supp(f)=\emptyset$ since $f$ is an $A$-flow, and the result follows.
  \end{proof}
  
  \begin{lemma}\label{lem:intflowdec}
    Let  $G$ be a graph and $k\ge 2$ be an integer. Any $k$-flow $f$ in $G$ is the sum of $k$-flows  $(f_i)_{1\le i \le s}$, whose supports
    are cycles of $G$, and such that for any $1\le j \le s$, and any edge $e$ of $G$, $\sum_{1\le i \le j} f_i(e)$ lies between $0$ and $f(e)$.
  \end{lemma}
  
  \begin{proof}
   Let $f$ be a $k$-flow in $G$. Up to reversing the orientations of a subset of edges $e$ of $G$ (and replacing $f(e)$ by $-f(e)$) we can assume that $f(e)\ge 0$ for every edge $e$ of $G$.
   As long as $\supp(f)$ contains a directed cycle $C$, subtract from $f$ a flow of value 1 along $C$. Note that for each edge $e$, the sequence of flow values on $e$ during this procedure is non-increasing and all values remain non-negative. When $\supp(f)$ has no directed cycle, $\supp(f)=\emptyset$ since $f$ is a $k$-flow with non-negative flow values.
  \end{proof}

Lemmas \ref{lem:linear} and \ref{lem:intflowdec} imply that the difference $g-f$ of two flows $g$ and $f$ can be written as a sum of flows whose supports are cycles. Therefore, $g$ can be obtained from $f$ by successively adding flows whose supports are cycles. The main question considered in this paper is: \emph{when $f$ and $g$ are nowhere-zero, can we make sure that all intermediate flows between $f$ and $g$ are also nowhere-zero?}

\medskip

Note that any nowhere-zero $k$-flow $f$ in a graph $G$ can be turned into
a nowhere-zero $\mathbb{Z}_k$-flow of $G$ by replacing $f(e)$ by $f(e)
\pmod k$ for any edge $e$. 
It was proved by Tutte that conversely, every nowhere-zero $\mathbb{Z}_k$-flow
in $G$ can be turned into a nowhere-zero $k$-flow.

  \begin{theorem}[\cite{Tutte190}]\label{thm:tutteintmod}
    For any graph $G$, integer $k$, and 
    $\mathbb{Z}_k$-flow $f$ of $G$, there is a $k$-flow
    $g$ in $G$ such that for every $e\in E(G)$, $f(e)\equiv g(e)\pmod k$. In particular, $f$ is nowhere-zero if and only if $g$ is nowhere-zero.
  \end{theorem}

Using the property that
the number of nowhere-zero $A$-flows in $G$ is a polynomial that only
depends on $|A|$, he obtained the following result.

    \begin{theorem}[\cite{Tutte190}]\label{thm:tuttecard}
    For any graph $G$ and abelian group $A$, $G$ admits a nowhere-zero $A$-flow if and only if it admits a nowhere-zero $|A|$-flow.
  \end{theorem}

  Note that a nowhere-zero $k$-flow is also a
  nowhere-zero $k'$-flow for any $k'>k$, and thus the result above
  implies that if a graph $G$ has a nowhere-zero $A$-flow for some abelian group
  $A$, then it has nowhere-zero $B$-flow for all abelian groups $B$
  with $|B|\ge |A|$. 

\subsection{Flow reconfiguration} Given a graph $G$ and an abelian group
$A$, the \emph{reconfiguration graph} $\mathcal{F}(G, A)$ is defined
as follows: the 
vertices are all the nowhere-zero $A$-flows, and two flows $f$ and $g$ are
adjacent if the support of their difference $f-g$ is a cycle (a
2-regular connected subgraph of $G$). The \emph{reconfiguration graph}
$\mathcal{F}(G,k)$ is defined similarly, by considering all
nowhere-zero $k$-flows of $G$. Our main topic of interest is to
understand for which group $A$ or integer $k$ the reconfiguration
graphs  $\mathcal{F}(G, A)$ and $\mathcal{F}(G, k)$ are
connected. Note that when $G$ has no nowhere-zero $A$-flow or $k$-flow
these graphs are empty (and in particular connected).

\smallskip

In view of Theorem \ref{thm:tuttecard} above, a natural question is
whether the connectivity of $\mathcal{F}(G, A)$ and $\mathcal{F}(G,
|A|)$ are related. Let $f$ and $f'$ be two nowhere-zero $\mathbb{Z}_k$-flows in $G$. By
Theorem \ref{thm:tutteintmod}, there exist nowhere-zero $k$-flows
    $g$ and $g'$ in $G$ such that for every $e\in E(G)$, $f(e)\equiv
    g(e)\pmod k$ and $f'(e)\equiv g'(e)\pmod k$.
    We observe  the following.

\begin{observation}\label{obs:intgroup}
If $g$ and $g'$ are in the same connected component of $\mathcal{F}(G,k)$, then $f$ and $f'$ are in the same connected component of $\mathcal{F}(G, \mathbb{Z}_k)$. In particular, if $\mathcal{F}(G, k)$ is
connected, then $\mathcal{F}(G, \mathbb{Z}_k)$ is also connected.
\end{observation}

\begin{proof}
Assume that there exists a path $g_1, g_2, \ldots, g_s$
between $g_1=g$ and $g_s=g'$ in $\mathcal{F}(G, k)$. For every $1\le i
\le s$, consider the nowhere-zero $\mathbb{Z}_k$-flow $f_i$ in $G$
with $f_i(e)\equiv g_i(e) \pmod k$ for any $e\in E(G)$. Note that by definition, $f_1=f$
and
$f_s=f'$. Moreover $\supp(f_i-f_{i+1})=\supp(g_i-g_{i+1})$ for any
$1\le i < s$, and thus any two flows $f_i,f_{i+1}$ are adjacent in
$\mathcal{F}(G, \mathbb{Z}_k)$. It follows that there is a path
between $f$ and $f'$ in $\mathcal{F}(G, \mathbb{Z}_k)$, as desired.
\end{proof}

We note that for every two isomorphic abelian groups $A$ and $B$ and
for any graph $G$, $\mathcal{F}(G, A)$ and $\mathcal{F}(G, B)$ are
isomorphic. In particular one is connected if and only if the other is
connected. By the Chinese remainder theorem, this shows that if
$k_1,\ldots,k_s$ are pairwise coprime and $k=\prod_{i=1}^sk_i$, then
$\mathcal{F}(G, \mathbb{Z}_{k})$ is connected if and only if $\mathcal{F}(G, \mathbb{Z}_{k_1}\times \cdots \times
\mathbb{Z}_{k_s})$ is connected. In particular, when $\mathcal{F}(G, k)$ is
connected, $\mathcal{F}(G, \mathbb{Z}_{k_1}\times \cdots \times
\mathbb{Z}_{k_s})$ is also connected by Observation \ref{obs:intgroup}.

\medskip

When the groups $A$ and $B$ have the same cardinality but are not
isomorphic, it is not true in general that the connectivity of
$\mathcal{F}(G, A)$ is equivalent to the connectivity of
$\mathcal{F}(G, B)$. For instance we will see that $\mathcal{F}(K_4, \mathbb{Z}_4)$ is not connected
(Section \ref{sec:example}) while $\mathcal{F}(K_4, \mathbb{Z}_2\times \mathbb{Z}_2)$ is
connected (Section \ref{sec:z2z2}). 
This is in stark contrast with
Theorem \ref{thm:tuttecard}, which states that the existence of a
nowhere-zero $A$-flow only depends on $|A|$.

\medskip

We believe that  the converse of Observation \ref{obs:intgroup}
also holds. 

\begin{problem}
Let $G$ be a graph and $k$ be an integer. Prove that if $\mathcal{F}(G,\mathbb{Z}_k)$ is
connected, then $\mathcal{F}(G, k)$ is also connected.
\end{problem}

We do not know how to prove the result even in the case $k=4$. 




\subsection{Recoloring}\label{sec:recol} In the following, we will use classical
results on graph recoloring. Given a graph $G$ and an integer $k$, the
graph $\mathcal{C}(G,k)$ is defined as follows: 
its vertices are all
the proper $k$-colorings of $G$, and two colorings are adjacent if and
only if they differ on exactly one vertex. 
For example, $\mathcal{C}(K_k,k)$ consists of $k!$ isolated
vertices, and in particular is not connected, for any $k\ge 2$.
In the following, we will need results on 
the reconfiguration graph $\mathcal{C}(G,k)$. One such example is Theorem
\ref{coloring_degeneracy} below.

\smallskip

A graph $G$ is \emph{$d$-degenerate} if every subgraph of $G$ contains a
vertex of degree at most $d$. We will use the following
classical result : 

\begin{theorem}[\cite{BonsmaCereceda2007, degeneracy2006}]\label{coloring_degeneracy}
    If $G$ is a $d$-degenerate graph, then for any integer $k \geq d + 2$,
    the graph $\mathcal{C}(G,k)$ is connected.
  \end{theorem}

In Section \ref{sec:z2z2} we will see another type of reconfiguration
graph, for edge-coloring this time, with a different elementary
operation.

\subsection{A useful example}\label{sec:example}

It can be checked that (any orientation of) $K_4$ has 6 distinct
nowhere-zero $\mathbb{Z}_4$-flows (one of them is given in Figure~\ref{figure:k4z4flow}), so $\mathcal{F}(G,\mathbb{Z}_4)$ has
6 vertices.
Indeed, consider a nowhere-zero $\mathbb{Z}_4$-flow $f$ in $K_4$. Observe that
each vertex must be incident to an edge $e$ with $f(e)=2$, since no
three elements of $\{-1,1\}$ can sum to $0 \pmod 4$. Moreover, each
vertex must be incident to a single such edge $e$, since otherwise the
conservation rule would be violated again. This shows that the edges
with flow value 2 form a perfect matching in $K_4$, and then the other flow values must be organized as another perfect matching with value 1 and the last perfect matching with value 3. We can then
observe that for any directed cycle $C$ that intersects an edge with flow value
2, the values $1,2,3$ must appear on $C$, so no flow can be added to
$C$ without creating a zero edge. It follows that the only cycle in
which some flow can be added is the 4-cycle consisting of the edges
with flow value distinct from 2 (see Figure~\ref{figure:k4z4flow} for an example). It can be checked that this is an
involution, so $\mathcal{F}(G,\mathbb{Z}_4)$ is a perfect matching on
6 vertices (and in particular it is not connected). We will see at the end of Section \ref{sec:z2z2} that this holds more generally for all uniquely 3-edge-colorable cubic planar graphs.

\begin{figure}
\hfill
\begin{subfigure}{0.45\textwidth}
\centering
\includegraphics[scale=.7, page=1]{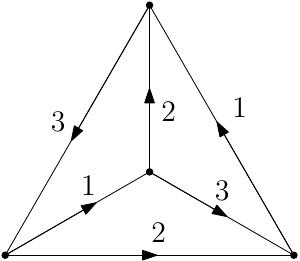}
\subcaption{}
\end{subfigure}
\hfill
\begin{subfigure}{0.45\textwidth}
\centering
\includegraphics[scale=.7, page=2]{k4z4}
\subcaption{}
\end{subfigure}\hfill

\caption{Illustration of a $\mathbb{Z}_4$-flow in $K_4$ and the only cycle along which some flow value can be added.}\label{figure:k4z4flow}
\end{figure}

\smallskip

We can then compute the reconfiguration graph $\mathcal{F}(K_4, 4)$ of
the nowhere-zero 4-flows of $K_4$, depicted in Figure
\ref{figure:k_4_integer_4_flow}. It can be observed that each vertex of
$\mathcal{F}(K_4,\mathbb{Z}_4)$ corresponds to 4 pairwise adjacent nowhere-zero
$4$-flows (all of which are equal to the same nowhere-zero
$\mathbb{Z}_4$-flow if we take all flow values modulo 4 : choose value $+2$ or $-2$ for each edge with value $2 \pmod 4$, and this determines all other values). As a result, each
connected component of $\mathcal{F}(K_4, 4)$ corresponds to one edge
of the matching of $\mathcal{F}(K_4, \mathbb{Z}_4)$, so
$\mathcal{F}(K_4, 4)$ consists of 3 components of size 8 (and in
particular it is not connected).

\begin{figure}[ht]
    \centering
    \includegraphics[width=\textwidth]{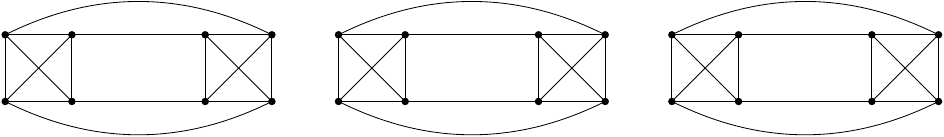}
    \caption{The reconfiguration graph $\mathcal{F}(G, 4)$} of
    nowhere-zero $4$-flows in $K_4$.
    \label{figure:k_4_integer_4_flow}
\end{figure}

\section{Frozen configurations}\label{sec:frozen}

A typical way to show that the reconfiguration graph
$\mathcal{C}(G,k)$ 
defined above for colorings is not connected is to
find a proper $k$-coloring of $G$ where for each vertex $v$, all the colors
distinct from that of $v$ appear in the neighborhood of $v$. In this
case no vertex
can be recolored, and the coloring is said to be
\emph{frozen}. Equivalently, the
corresponding vertex in $\mathcal{C}(G,k)$ 
is isolated. Assuming $\mathcal{C}(G,k)$
has at least two vertices, this
implies that these graphs are  not
connected.

It is tempting to try to use similar techniques in the setting of flow
reconfiguration. We now show that if the abelian group $A$ is
sufficiently large, no nowhere-zero $A$-flow can be frozen (in other
words, $\mathcal{F}(G,A)$ has minimum degree at least one), so in this case techniques based on frozen configurations
are inefficient for producing lower bounds in flow reconfiguration.

\smallskip

We need the following simple lemma.

\begin{lemma}\label{colorings_min_colors_connected}
    Let $G$ be a graph with minimum degree at least $3$. If the edges
    of  $G$ are colored (not necessarily properly) using $\ell\ge 3$ colors, then there exists a cycle in $G$ that does not contain any edge of the least frequent color.
\end{lemma}

\begin{proof}
    Let $n= |V(G)|$ and $m=|E(G)|$. Since the minimum degree of $G$ is
    at least $3$, we have $2m = \sum_{v \in V(G)} d(v) \geq 3n$, and
    thus $ m \geq \frac{3n}{2}$. Let $c$ be the least frequent of the
    $k$ 
    colors (possibly, $c$ does not appear at all). As the $m$ edges are partitioned into $\ell \geq 3$ colors,
    the number of edges colored $c$ is at most $\frac{m}{\ell}$. It
    follows that there are at least \[m \cdot \left(1 - \frac{1}{\ell}\right)\ge
      \frac{2m}{3} \ge \frac23 \cdot \frac32\cdot n=n\] edges not colored with $c$.  A graph on $n$ vertices with at least $n$ edges must contain a cycle, therefore, the subgraph of $G$ induced by the edges not colored $c$ contains a cycle, which completes the proof.
  \end{proof}

  We obtain the following as a simple consequence.

\begin{theorem}\label{thm:frozenabel}
    For any 2-edge-connected graph $G$ and abelian group $A$ with $|A| \geq
    6$ or $A = \mathbb{Z}_2 \times \mathbb{Z}_2$, if the reconfiguration graph
    $\mathcal{F}(G, A)$ is non-empty, then it has minimum degree at least 1.
\end{theorem}

\begin{proof}
  Let $H$ be the graph obtained from $G$ by repeatedly suppressing all
  vertices of degree 2 (that is, replacing every path $P$ in which all
  internal vertices have degree two by an edge between the endpoints
  of $P$). Note that $H$ has minimum degree at least 3 and
  nowhere-zero $A$-flows in $H$ are in bijection with nowhere-zero
  $A$-flows in $G$. Consider a nowhere-zero $A$-flow in $G$, and the
  corresponding nowhere-zero $A$-flow $f$ in $H$. For each $e\in
  E(H)$, color $e$ with the set $\{f(e),-f(e)\}$. If $A=\mathbb{Z}_2
  \times \mathbb{Z}_2$, there are 3 possible colors ($01$, $10$ and
  $11$), and if $|A|\ge 6$ there are at least $\lceil
  \tfrac12(|A|-1)\rceil\ge 3$ possible colors. By Lemma
    \ref{colorings_min_colors_connected}, there is a non-zero element
    $a\in A$ and a cycle $C$ in $H$ (and thus in $G$), on which
    no edge has flow value $\pm a$. We can then add a flow whose
    support is $C$ with flow values $\pm a$, and the resulting flow remains
    nowhere-zero. This shows that every vertex of $\mathcal{F}(G, A)$
    has degree at least 1.
  \end{proof}
  
  We note that when $|A|\ge 7$, the proof of Theorem \ref{thm:frozenabel} can be
refined to show that the minimum degree of 
$\mathcal{F}(G, A)$  is
at least linear in the number of vertices of degree at least 3 in
$G$.

\medskip

It was pointed out to us by Zolt\'an Szigeti that the proof of Theorem \ref{thm:frozenabel} can be extended to work also in the case $A=\mathbb{Z}_4$, by observing that in a minimum counterexample, the edges with flow value $2 \pmod 4$ form a matching, and thus there are at most $n/2$ such edges. It follows that there are at least $3n/2-n/2=n$ edges with flow values $\pm 1 \pmod 4$, and thus there is a cycle $C$ whose flow values are $\pm 1 \pmod 4$ (we can then safely add $2 \pmod 4$ along $C$, and the resulting $\mathbb{Z}_4$-flow is still nowhere-zero). 

\medskip

The proof of the variant of Theorem \ref{thm:frozenabel} for integer flows is surprisingly more direct than in the group case, and works for all $k\ge 2$. 

  \begin{theorem}\label{thm:frozenint}
    For any 2-edge-connected graph $G$ and integer $k\ge 2$, if the
    reconfiguration graph
    $\mathcal{F}(G, k)$ is non-empty, then it has minimum degree at least 1.
\end{theorem}

\begin{proof}
    Consider a nowhere-zero $k$-flow of $G$, and assume without loss of generality (up to reversing some edges $e$ and replacing their flow value $f(e)$ by $-f(e)$) that all the edges of $G$ have positive flow value. Then in the underlying orientation of $G$, every vertex has non-zero out-degree, and in particular there is a directed cycle $C$. The values appearing on $C$ are in the set $\{1,\ldots,k-1\}$, so we can subtract $k$ to all flow values along $C$ and the resulting flow is still a nowhere-zero $k$-flow. It is adjacent to $f$ in  $\mathcal{F}(G, k)$, as the two flows only differ on $C$.
\end{proof}

Theorems \ref{thm:frozenabel} and the subsequent remark on $\mathbb{Z}_4$-flows show that except possibly when $|A|\le 3$ or $A=\mathbb{Z}_5$, no frozen nowhere-zero $A$-flow exists. When $|A|\le 3$ it is easy to construct frozen flows, so only the case of $\mathbb{Z}_5$-flows remains. 

\medskip

In the remainder of this section, we provide an example of a graph $G$ such that the flow reconfiguration graph $\mathcal{F}(G,\mathbb{Z}_5)$  contains an isolated vertex (a frozen nowhere-zero flow). Since by multiplying every flow-value by a non-zero element one still retains a flow, every graph with some nowhere-zero $\mathbb{Z}_5$-flow also contains at least two distinct nowhere-zero $\mathbb{Z}_5$-flows, and so this indeed shows that $\mathcal{F}(G,\mathbb{Z}_5)$ is not connected. By Observation~\ref{obs:intgroup}, we then also have that $\mathcal{F}(G,5)$ is not connected.

\begin{theorem}\label{thm:Z5frozen}
There exists a graph $G$ with a nowhere-zero $\mathbb{Z}_5$-flow which forms an isolated vertex in $\mathcal{F}(G,\mathbb{Z}_5)$. In particular, $\mathcal{F}(G,\mathbb{Z}_5)$ and $\mathcal{F}(G,5)$ are disconnected.
\end{theorem}

Instead of directly showing the example demonstrating the theorem, let us give an equivalent characterization of frozen nowhere-zero $\mathbb{Z}_5$-flows which was instrumental in finding the example in Theorem~\ref{thm:Z5frozen}. Note that since $4=-1$ and $3=-2$ in $\mathbb{Z}_5$, possibly by suitably changing the underlying orientation it is clear that we may w.l.o.g.\ restrict ourselves to looking for frozen $\mathbb{Z}_5$-flows taking only values $1$ or $2$. 

\begin{lemma}\label{lem:Z5frozen}
Let $G$ be a graph and $D$ be an orientation of $G$. Let $f:E(G)\rightarrow \{1,2\}\subseteq \mathbb{Z}_5$ be a $\mathbb{Z}_5$-flow in $G$ with respect to $D$. Then $f$ is an isolated vertex of $\mathcal{F}(G,\mathbb{Z}_5)$ if and only if there exist maps $\eta_1, \eta_2:V(G)\rightarrow \mathbb{Z}$ such that the following hold.
\begin{itemize}
    \item For each $i\in \mathbb{Z}$, the induced subgraphs $G[\eta_1^{-1}(i)]$ and $G[\eta_2^{-1}(i)]$ are forests.
    \item For every arc $(u,v)\in D$ with $f(uv)=1$, we have $\eta_1(u)=\eta_1(v)$ and $\eta_2(u)<\eta_2(v)$.
    \item For every arc $(u,v)\in D$ with $f(uv)=2$, we have $\eta_1(u)<\eta_1(v)$ and $\eta_2(u)=\eta_2(v)$.
\end{itemize}
\end{lemma}
\begin{proof}
\noindent
\begin{itemize}\item[$\Rightarrow$] Suppose first that $f$ is an isolated vertex of $\mathcal{F}(G,\mathbb{Z}_5)$, and let us show that then maps $\eta_1, \eta_2$ with the desired properties exist.  
In the following, let $E_1, E_2$, respectively, denote the set of edges of $G$ with flow value $1$ and $2$, respectively. We then claim that $G_1:=(V(G),E_1)$ and $G_2:=(V(G),E_2)$ are forests. Indeed, suppose towards a contradiction that there exists a cycle $C$ in $G$ all whose corresponding arcs in $D$ have the same flow value $i\in \{1,2\}\subseteq \mathbb{Z}_5$. Let $C^+, C^-$ be a partition of $E(C)$ into ``forward''- and ``backward''-edges in the orientation according to $D$ and let $f':E(G)\rightarrow \mathbb{Z}_5$ be defined by $f'(e):=f(e)$ for every $e\notin E(C)$, $f'(e):=f(e)+2i$ for every $e\in C^+$ and $f'(e):=f(e)-2i$ for every $e\in C^-$. It is then easy to see that $f'$ is still nowhere-zero and adjacent to $f$ in $\mathcal{F}(G,\mathbb{Z}_5)$, a contradiction to our assumption that $f$ forms an isolated vertex, and hence $G_1$ and $G_2$ are indeed forests. We next claim that the directed graphs $D_1:=D/E_1$ and $D_2:=D/E_2$ formed from $D$ by contracting all arcs in $E_1$ or $E_2$, respectively, are acyclic (contain no directed cycles). Indeed, suppose towards a contradiction that there exists a directed cycle in $D_i$ for some $i\in \{1,2\}$. Then 
by decontracting the vertices along this directed cycle into oriented paths with edges in $E_i$ we find that there exists a cycle $C$ in $G$ with vertex-sequence $v_0,v_1,\ldots,v_\ell=v_0$ such that for every $1\le j \le \ell$ with $v_{j-1}v_j\in E_{3-i}$ we have that the edge $v_{j-1}v_j$ is oriented from $v_{j-1}$ to $v_j$ in $D$. However, we can now define a new $\mathbb{Z}_5$-flow $f'$ on $G$ by setting $f'(e):=f(e)$ for all $e\notin E(C)$, and $f'(v_{j-1}v_j):=f(v_{j-1}v_j)+3-i$ if $(v_{j-1},v_j)\in D$ as well as $f'(v_{j-1}v_j):=f(v_{j-1}v_j)-3+i$ if $(v_j,v_{j-1})\in D$. It is not hard to check that this flow is again nowhere-zero, contradicting our assumption that $f$ was an isolated vertex in $\mathcal{F}(G,\mathbb{Z}_5)$.

Having established that $D_1, D_2$ are acyclic, we conclude that they admit topological orderings. Translating this back to $G$, we find that there is a linear ordering $U_1,\ldots,U_s$ of the connected components of $G_1$ and a linear ordering $V_1,\ldots,V_t$ of the connected components of $G_2$ such that for every $(u,v)\in D$ with $uv\in E_1$ we have $u\in V_i, v\in V_j$ for some $1\le i<j\le s$ and analogously for every $(u,v)\in D$ with $uv\in E_2$ we have $u\in U_i, v\in U_j$ for some $1\le i<j\le t$.

We now define $\eta_1, \eta_2:V(G)\rightarrow \mathbb{Z}$ by setting $\eta_1(v):=i$ if and only if $v\in U_i$ and $\eta_2(v):=j$ if and only if $v\in V_j$. It remains to check that the three conditions required by the lemma statement are satisfied by $\eta_1,\eta_2$. Indeed, for each $k\in \mathbb{Z}$ and $i\in \{1,2\}$, we have that each of $\eta_i^{-1}(k)$ is either empty or a component of $G_i$. Since $G_i$ is a forest, the only way how $G[\eta_i^{-1}(k)]$ could contain a cycle is if that cycle would use an edge in $E_{3-i}$ between two vertices in the connected component of $G_i$ defined by $\eta_i^{-1}(k)$. However, this would then create a directed loop and hence a directed cycle in  $D_i$, a contradiction since we showed the latter is acyclic. Hence, indeed $G[\eta_1^{-1}(k)],G[\eta_2^{-1}(k)]$ are forests for every $k\in \mathbb{Z}$, as desired. 

Next, consider any arc $(u,v)\in D$ with $f(uv)=i\in \{1,2\}$. Then $uv\in E_i$ and so $u,v$ form part of the same connected component of $G_i$ and hence we have $\eta_i(u)=\eta_i(v)$ by definition of $\eta_i$. By the properties of the orderings $U_1,\ldots,U_s$ and $V_1,\ldots,V_t$ of the components of $G_1$ and $G_2$ we furthermore must have that the arc $(u,v)$ satisfies $\eta_{3-i}(u)<\eta_{3-i}(v)$, as claimed. This concludes the proof of the $\Rightarrow$-direction of the equivalence.
\item[$\Leftarrow$] Suppose $\eta_1,\eta_2:V(G)\rightarrow \mathbb{Z}$ satisfy the three conditions stated in the lemma and let us show that $f$ is an isolated vertex in $\mathcal{F}(G,\mathbb{Z}_5)$. Towards a contradiction, suppose there was a neighbor $f'$ of $f$ in $\mathcal{F}(G,\mathbb{Z}_5)$ and let $C$ be the unique cycle in $G$ which forms the support of the difference flow $f'-f$. Let $v_0,v_1,\ldots,v_\ell=v_0$ be a cyclic order of the vertices on $C$. We then immediately obtain that there exists some element $x\in \mathbb{Z}_5\setminus \{0\}$ such that $f'(e)=f(e)$ for all $e\notin E(C)$ and $f'(v_{j-1}v_j)=f(v_{j-1}v_j)+x$ for every $j$ such that $(v_{j-1},v_j)\in D$ and $f'(v_{j-1}v_j)=f(v_{j-1}v_j)-x$ for every $j$ such that $(v_{j},v_{j-1})\in D$. Possibly by reversing the cyclic order we may w.l.o.g.\ assume that $x\in \{1,2\}$. Since $\eta_1^{-1}(i),\eta_2^{-1}(i)$ induce forests in $G$ for every $i\in \mathbb{Z}$, it directly follows that each of $\eta_1, \eta_2$ takes on at least two distinct values on $C$. In particular, there must exist $j_1, j_2 \in \{1,\ldots,\ell\}$ such that $\eta_1(v_{j_1})<\eta_1(v_{j_1-1})$ and $\eta_2(v_{j_2})<\eta_2(v_{j_2-1})$. By the second and third property satisfied by $\eta_1, \eta_2$ we then find that we must have $f(v_{j_1-1}v_{j_1})=2$ and $(v_{j_1},v_{j_1-1})\in D$ as well as $f(v_{j_2-1}v_{j_2})=1$ and $(v_{j_2},v_{j_2-1})\in D$. This in turn implies that $f'(v_{j_1-1}v_{j_1})=2-x$ and $f'(v_{j_2-1}v_{j_2})=1-x$. Since $x\in \{1,2\}$, one of these two values must be $0$. This is the desired contradiction since $f'$ was initially assumed to be a nowhere-zero $\mathbb{Z}_5$-flow. This concludes the proof of the $\Leftarrow$-direction of the equivalence.
\end{itemize}
\end{proof}

\begin{figure}[htb]
    \centering
    \includegraphics[scale=1]{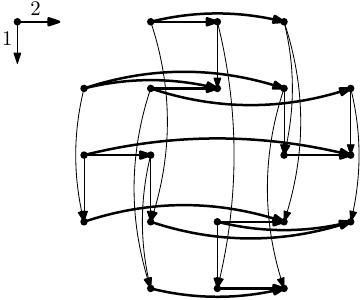}
    \caption{A graph $G$ with a frozen nowhere-zero $\mathbb{Z}_5$-flow.}
    \label{fig:Z5frozen}
\end{figure}

In order to prove Theorem~\ref{thm:Z5frozen}, it remains to show that a graph $G$ and flow $f$ as in Lemma~\ref{lem:Z5frozen} exist. A particular instance of a graph as in the lemma that is natural to look for is if the vertex-set of $G$ is identified with a finite set of points on the integer grid $\mathbb{Z}\times \mathbb{Z}$ and the functions $\eta_1, \eta_2$ simply map vertices to their respective $x$- and $y$-coordinates. The conditions in the lemma can then be reformulated as saying that every column and every row of the grid should induce a forest, and that all $1$-arcs go in the same direction along columns while all $2$-arcs go in the same direction along rows.

Indeed, with the generous help of Yuval Wigderson and Patryk Morawski and the tool AlphaEvolve~\cite{Novikov2025} we managed to find such an example of a frozen $\mathbb{Z}_5$-flow on $20$ vertices, which is depicted in Figure~\ref{fig:Z5frozen}. The validity of the flow conditions can easily be verified by hand from the figure, as well as the fact that the remaining conditions of Lemma~\ref{lem:Z5frozen} are satisfied.

\medskip

\noindent \emph{Proof of Theorem \ref{thm:Z5frozen}.} Consider the $\mathbb{Z}_5$-flow of Figure \ref{fig:Z5frozen} and apply Lemma \ref{lem:Z5frozen}. \hfill $\Box$

\section{The case of \texorpdfstring{$\mathbb{Z}_2\times \mathbb{Z}_2$}{Z2 times Z2}}\label{sec:z2z2}

In the introduction we have seen one of the few possible ways to define
coloring reconfiguration, which will be crucial in the next section to
prove results on flow reconfiguration in planar graphs by duality.
In this section we relate $(\mathbb{Z}_2\times \mathbb{Z}_2)$-flow
reconfiguration to the reconfiguration of edge-colorings using a
different type of elementary operation, \emph{Kempe
changes}, which we define below. It should be noted that, in the literature on Kempe change recoloring (unlike  for single-edge or single-vertex recoloring which we have seen in Section \ref{sec:recol} and which will be the topic of the next section), edge-colorings are
considered as \emph{unlabelled} partitions of 
 the edge set, or equivalently they are considered up to color permutation.  Note that every unlabelled edge-coloring of a graph $G$ with $k$ unlabelled colors corresponds to 
 $k!$ distinct (labelled) $k$-edge-colorings of $G$.  \textbf{In the remainder of this section, all edge-colorings and color classes are considered to be unlabelled, unless specified otherwise.}

\smallskip

Given a proper edge-coloring $c$ in a graph $G$, and two colors $\{i,j\}$
of $c$, a \emph{Kempe chain}
is a connected component of the subgraph of $G$ spanned by the edges
colored $i$ or $j$. Note that since $c$ is a proper edge-coloring, 
such a subgraph  has maximum degree at most 2 and is bipartite, so a
Kempe chain is a  path or an even cycle. Consider a Kempe chain with
colors $\{i,j\}$ in $G$, and let $c'$ be the edge-coloring of 
$G$ obtained from
$c$ by exchanging colors $i$ and $j$ in the Kempe chain. Note that
$c'$ is still a proper edge-coloring of $G$. We say that $c'$ has been
obtained from $c$ by a \emph{Kempe change}. Let
$\mathcal{K}(G,k)$ be the graph whose vertices are all the (unlabelled) proper
$k$-edge-colorings, with an edge between two colorings if and
only if they
differ by a Kempe change. See \cite{Mohar} for more background on
coloring reconfiguration with respect to Kempe changes.



\smallskip

Our main result in this section is the following.

\begin{theorem}\label{thm:z2z2}
If $G$ is a  cubic graph, then $\mathcal{F}(G,\mathbb{Z}_2\times
\mathbb{Z}_2)$  is connected if and only if $\mathcal{K}(G,3)$ is
connected.
\end{theorem}

Before proving the theorem, we mention a few relevant properties of
nowhere-zero $(\mathbb{Z}_2\times
\mathbb{Z}_2)$-flows and their relations with 3-edge-colorings in
cubic graphs.

Let us write the elements of $\mathbb{Z}_2\times
\mathbb{Z}_2$ as $00,01,10,11$, and note that $a=-a$ for any $a\in \mathbb{Z}_2\times
\mathbb{Z}_2$, so in particular edge orientations are irrelevant when
considering flow values defined in $\mathbb{Z}_2\times
\mathbb{Z}_2$. Another consequence is that for any vertex $v$ of degree 3
in a graph $G$ with a nowhere-zero $(\mathbb{Z}_2\times
\mathbb{Z}_2)$-flow, each flow value appears at most once on the edges
incident to $v$, and in particular each of these three edges carries a
different flow value. As a consequence, any nowhere-zero $(\mathbb{Z}_2\times
\mathbb{Z}_2)$-flow of a cubic graph $G$ is also a proper
3-edge-coloring of $G$. Conversely, any proper
3-edge-coloring of a cubic graph $G$ can be turned into $3!=6$ nowhere-zero $(\mathbb{Z}_2\times
\mathbb{Z}_2)$-flows of $G$.

Next, consider a nowhere-zero $(\mathbb{Z}_2\times
\mathbb{Z}_2)$-flow $f$ in a cubic graph $G$. Let $C$ be a cycle and
write $\mathbb{Z}_2\times
\mathbb{Z}_2\setminus \{00\}=\{a,b,c\}$. Assume that adding the flow value
$a$ along
$C$ to $f$ produces a nowhere-zero $(\mathbb{Z}_2\times
\mathbb{Z}_2)$-flow $g$. Then $f(C)=\{b,c\}$ and the values $b$ and
$c$ alternate along $C$ (in particular $C$ is an even cycle). As
$a+b+c=00$ we have 
$b+a=c$ and $c+a=b$, so that in $g$, the values $b$ and $c$ along $C$ have been
exchanged. This shows that if we view $f$ and $g$ as 3-edge-colorings,
they are related by a Kempe change. Conversely, given two proper
3-edge-colorings $c$ and $c'$ of $G$ related by a Kempe change, we can assign
values from $\mathbb{Z}_2\times
\mathbb{Z}_2\setminus \{00\}$ bijectively to the color classes of $c$
and $c'$ such that the corresponding nowhere-zero $(\mathbb{Z}_2\times
\mathbb{Z}_2)$-flows are adjacent in $\mathcal{F}(G,\mathbb{Z}_2\times
\mathbb{Z}_2)$. We write this as an observation for future reference.

\begin{observation}\label{obs:kempe}
  Let $G$ be a cubic graph. For any edge $f_1f_2$ in  $\mathcal{F}(G,\mathbb{Z}_2\times
\mathbb{Z}_2)$, the 3-edge-colorings corresponding to $f_1$ and $f_2$
are adjacent in $\mathcal{K}(G,3)$. Conversely, for any edge  $c_1c_2$ in $\mathcal{K}(G,3)$, there is an edge $f_1f_2$ in  $\mathcal{F}(G,\mathbb{Z}_2\times
\mathbb{Z}_2)$ such that for each $i=1,2$, $c_i$ is the 3-edge-coloring of $G$
corresponding to $f_i$. 
\end{observation}

We can now prove the following, which immediately implies  Theorem \ref{thm:z2z2}.

\begin{theorem}\label{thm:z2z2b}
Let $G$ be a cubic graph. Let $f,g$ be nowhere-zero $(\mathbb{Z}_2\times
\mathbb{Z}_2)$-flows in $G$, and let $c$ and $d$, respectively, be the corresponding
3-edge-colorings of $G$. Then $f$ and $g$ are in the same connected
component of $\mathcal{F}(G,\mathbb{Z}_2\times
\mathbb{Z}_2)$ if and only if $c$ and $d$ are in the same connected
component of $\mathcal{K}(G,3)$.
\end{theorem}

\begin{proof} Let $G$ be a
  2-edge-connected cubic graph.

  We first show that given any proper
  3-edge-coloring $c$ of $G$, the corresponding $3!=6$ nowhere-zero $(\mathbb{Z}_2\times
\mathbb{Z}_2)$-flows of $G$ are in the same connected component of $\mathcal{F}(G,\mathbb{Z}_2\times
\mathbb{Z}_2)$.
Consider two  nowhere-zero $(\mathbb{Z}_2\times
\mathbb{Z}_2)$-flows $f$ and $g$ that induce the same edge
partition up to relabeling. Then there is a permutation $\pi : \mathbb{Z}_2\times
\mathbb{Z}_2
\setminus \{00\} \to \mathbb{Z}_2\times
\mathbb{Z}_2\setminus \{00\}$ such that for every edge $e \in
E(G)$, we have $g(e) = \pi(f(e))$. Since any permutation can be
expressed as a product of transpositions, it suffices to consider the
case where $g
$ is obtained from $f$ by swapping two elements $a,
b$, while the third non-zero value remains unchanged.
Consider the  flow $f + g$, and observe that $(f+g)(e)$ is nonzero
only when $f(e)$ and $g(e)$ differ, in which case
$(f+g)(e)=a+b$. Therefore, the support of $f+g$ consists of a union of
edge-disjoint cycles $C_1,\ldots, C_k$ on which the flow $f+g$ is equal
to $a+b$. For every $1\le i \le k$, we now add to $f$ a flow of
value $a+b$ along the cycle $C_i$. Note that the resulting flow is
precisely $g$, and all the intermediate flows are
nowhere-zero, as the cycles $C_i$ are non-zero. This shows that $f$ and $g$ are in the same connected
component of $\mathcal{F}(G,\mathbb{Z}_2\times
\mathbb{Z}_2)$, as claimed.

\smallskip

Now, consider two 3-edge-colorings $c,d$ of $G$, and assume that  two vertices $f$ and $g$ of $\mathcal{F}(G,\mathbb{Z}_2\times
\mathbb{Z}_2)$ associated to $c$ and $d$, respectively (in the
sense that $c$ is the 3-edge-coloring corresponding to $f$ and $d$ is
the 3-edge-coloring corresponding to $g$), are
connected by a path $P=f_1,\ldots,f_k$ in $\mathcal{F}(G,\mathbb{Z}_2\times
\mathbb{Z}_2)$, with $f=f_1$ and $g=f_k$. For
any $1\le i \le k$, let $c_i$ be the 3-edge-coloring associated to
$f_i$. Note
that by Observation \ref{obs:kempe}, any two vertices $c_i,c_{i+1}$
are adjacent in $\mathcal{K}(G,3)$, and thus $c=c_1$ and $d=c_k$ are in the
same component of $\mathcal{K}(G,3)$.

\smallskip

Conversely, consider two vertices $f,g$ of $\mathcal{F}(G,\mathbb{Z}_2\times
\mathbb{Z}_2)$, and assume that the corresponding 3-edge-colorings of $G$ are
connected by a path $P=c_1,\dots, c_k$ in $\mathcal{K}(G,3)$, where
$c_1$ is the coloring corresponding  to $f$, and
$c_k$ is the coloring corresponding to $g$. Recall that the colorings
$c_i$ are considered as unlabelled partitions of the edge set of $G$
into 3 parts. It follows from Observation \ref{obs:kempe} that for each
edge $c_ic_{i+1}$ in $P$, there is an edge $f_i^+f_{i+1}^-$ in  $\mathcal{F}(G,\mathbb{Z}_2\times
\mathbb{Z}_2)$ such that $c_i$ is the 3-edge-coloring of $G$
corresponding to $f_i^+$, and $c_{i+1}$ is the 3-edge-coloring of $G$
corresponding to $f_{i+1}^-$. Write $f_1^-=f$ and $f_k^+=g$. By the
paragraph above, for any $1\le i \le k$, $f_i^+$ and $f_i^-$ are
 in the same connected component of  $\mathcal{F}(G,\mathbb{Z}_2\times
\mathbb{Z}_2)$, since they correspond to the same 3-edge-coloring
$c_i$. It follows that $f$ and $g$ are in the same connected component
of $\mathcal{F}(G,\mathbb{Z}_2\times
\mathbb{Z}_2)$.
  \end{proof}

We emphasize that, while the statements of Theorems \ref{thm:z2z2} and \ref{thm:z2z2b} only
assume that $G$ is cubic, the results are trivial
when $G$ is not 3-edge-colorable (and in particular not 2-edge-connected),
since in this case the reconfiguration graphs are empty.

\medskip

It was proved by Belcastro and Haas \cite{Kempe_change_cubic} that for
every planar bipartite cubic graph $G$, the graph $\mathcal{K}(G,3)$
is connected (see also \cite{Fisk,Mohar}). Using Theorem \ref{thm:z2z2}, we immediately deduce the following.

\begin{corollary}\label{cor:z2z2}
    For any planar bipartite cubic graph $G$, $\mathcal{F}(G, \mathbb{Z}_2 \times \mathbb{Z}_2)$ is connected.
\end{corollary}

It was noticed by Mohar \cite{Mohar} that $\mathcal{K}(K_{3,3},3)$ is not
connected. By Theorem \ref{thm:z2z2}, this implies
that $\mathcal{F}(K_{3,3}, \mathbb{Z}_2 \times \mathbb{Z}_2)$ is not
connected, and  shows that the assumption of planarity is
necessary in Corollary \ref{cor:z2z2}. Similarly, it can be deduced
from a recent result of Florek \cite{Flo25} on Kempe equivalence in
proper 4-colorings of planar triangulations that if $G_n$ is a planar cubic graph
with two faces of degree $n$, and all the remaining faces have degree 5,
then $\mathcal{F}(G_n, \mathbb{Z}_2 \times \mathbb{Z}_2)$ has $\lfloor
n/6 \rfloor$ components (see Mohar \cite{Mohar} for the
connection between Kempe equivalence in proper 4-colorings of planar
triangulations and proper 3-edge-colorings of their dual graphs).

\medskip

We conclude this section with another simple but interesting
consequence of Theorem \ref{thm:z2z2}. A cubic graph $G$ is said to be a
\emph{Klee-graph}  if $G=K_4$, or if $G$ is obtained from a Klee-graph
$H$ by replacing a vertex by a triangle. More precisely, we consider a
vertex $v$ of $G$ with neighbors $u_1,u_2,u_3$, we remove $v$, add
three new pairwise adjacent vertices $v_1,v_2,v_3$, and finally we add
edges $u_iv_i$ for any $1\le i \le 3$. It is easy to check that
Klee-graphs have a single proper 3-edge-coloring (it is also known
\cite{fowler1998unique} that these are
the only cubic planar graphs with this property, but this is a
significantly harder result, and we will not need it), and in particular every two color classes form a Hamiltonian cycle. As a
consequence, $\mathcal{K}(G,3)$ is connected (it consists of a single
vertex). By Theorem \ref{thm:z2z2}, we obtain the following.

\begin{corollary}\label{cor:z2z2klee}
  For any
  Klee-graph $G$, $\mathcal{F}(G, \mathbb{Z}_2 \times \mathbb{Z}_2)$ is connected.
\end{corollary}

We note that more can be said in Corollary \ref{cor:z2z2klee}: using the
fact that any two color classes in a 3-edge-coloring of a Klee-graph form a Hamiltonian cycle, we can
prove that $\mathcal{F}(G, \mathbb{Z}_2 \times \mathbb{Z}_2)$ is the
complete bipartite graph $K_{3,3}$ (where the vertices correspond to
permutations of $\{01,10,11\}$ and the bipartition corresponds to the signs of
the permutations). We omit the details.

\medskip

We immediately obtain from Corollary \ref{cor:z2z2klee} that $\mathcal{F}(K_4, \mathbb{Z}_2 \times
\mathbb{Z}_2)$ is connected, which contrasts with the fact that
$\mathcal{F}(K_4, \mathbb{Z}_4)$ and $\mathcal{F}(K_4,4)$ are not
connected (see Section \ref{sec:example}). We now show that the observation of Section \ref{sec:example} on $K_4$ extends to all Klee-graphs.

\begin{observation}\label{obs:z4klee}
  For any
  Klee-graph $G$, $\mathcal{F}(G, \mathbb{Z}_4)$ is a perfect matching on 6 vertices.
\end{observation}

\begin{proof}
As we observed in Section \ref{sec:example}, for every nowhere-zero $\mathbb{Z}_4$-flow in a cubic graph $G$, the edges with flow value $2 \pmod 4$ form a perfect matching $M$ of $G$, and if we orient the remaining edges so that all flow values are equal to $1\pmod 4$, then every  vertex has in-degree 2 or out-degree 2 (in particular $E(G)-M$ is a collection of disjoint even cycles, and for each such cycle $C$, there are precisely two valid orientations of the edges of $C$. 

\smallskip

We now prove by induction on the number of vertices that the 3 color classes in the unique 3-edge-coloring of $G$ are the only perfect matchings of $G$ which are the support of the edges of flow value $2\pmod 4$ in some nowhere-zero $\mathbb{Z}_4$-flow of $G$. As the complement of each color class is a Hamiltonian cycle, there are exactly 2 valid orientations of this cycle, and thus exactly 2 nowhere-zero $\mathbb{Z}_4$-flows for each fixed perfect matching  of edges with flow value $2 \pmod 4$, which concludes the proof.

\smallskip

 \begin{figure}[htb] 
  \centering 
  \includegraphics[scale=1]{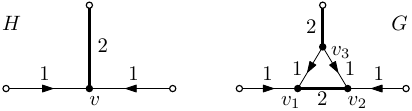}
  \caption{The nowhere-zero $\mathbb{Z}_4$-flows of $G$ are in bijection with the nowhere-zero $\mathbb{Z}_4$-flows of $H$. The case where $v$ has outdegree 2 is symmetric.}
  \label{fig:klee}
\end{figure}

The statement above is clear for $G=K_4$, as this graph has 3 perfect matchings and their complement is a 4-cycle. So we can assume that $G$ was obtained from some Klee-graph $H$ by replacing a vertex $v$ by a triangle $v_1v_2v_3$. But then observe that the nowhere-zero $\mathbb{Z}_4$-flows of $G$ are in bijection with the nowhere-zero $\mathbb{Z}_4$-flows of $H$ (see Figure \ref{fig:klee} for an illustration), and thus the result simply follows  by induction.
\end{proof}

\section{Planar duality}\label{sec:duality}

  A \emph{plane graph} is an embedding of a planar graph $G$ in the
  plane. The \emph{dual graph} $G^*$ of a plane graph $G$ is the plane graph whose
  vertices are the faces of $G$, and where we add an edge $e^*$ between two
  vertices $f_1^*$ and $f_2^*$ for each edge $e$ incident to the two
  faces $f_1$ and $f_2$ in $G$.  If the edges of $G$ are oriented,
  then we orient the edges of $G^*$ in a consistent way: if we look at an edge
  $e$ such that it is oriented towards the top, then the head of $e^*$ lies on the right of $e$, see Figure \ref{fig:dual}.
  
  \begin{figure}[htb] 
  \centering 
  \includegraphics[scale=1]{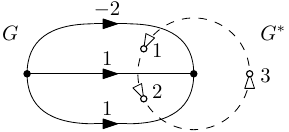}
  \caption{A plane (oriented) graph $G$ and its dual (oriented) graph $G^*$. Any proper coloring of $G^*$ induces a nowhere-zero flow in $G$.}
  \label{fig:dual}
\end{figure}

  \smallskip
  
In \cite{Tutte190}, Tutte established a fundamental result that
connects flows and colorings in plane graphs.

\begin{theorem}[\cite{Tutte190}]\label{equivalence_flow_color}
    Let $G$ be a $2$-edge-connected plane graph. Then $G$ has a
    nowhere-zero $k$-flow if and only if its dual graph $G^*$ has a proper $k$-coloring.
\end{theorem}

We now explain how colorings can be constructed from flows, and vice
versa, as it will be crucial in our reconfiguration setting.

    Let $G$ be a $2$-edge-connected plane graph (with underlying
    orientation $D$) and let $A$ be an
    abelian group. Let $c$ be a proper coloring of $G^*$ with colors
    in $A$. A nowhere-zero $A$-flow $f$ in $G$ is said to be
    \emph{induced by} $c$ if the following holds for every edge $e$ of
    $G$: assume that $e^*$ is oriented from $x^*$ to $y^*$
    in $G^*$, then 
    \begin{equation}\label{eq:finduc} f(e)=c(y^*)-c(x^*).\end{equation} See Figure \ref{fig:dual} for an illustration. Tutte proved that $f$ is
    indeed a nowhere-zero $A$-flow. Moreover, if $c$ is a proper $k$-coloring
    (with color values in $\{ 1, \ldots ,k \}$), then $f$ turns out to
    be a nowhere-zero $k$-flow.

Starting from a nowhere-zero $A$-flow $f$ in  $G$, a
proper coloring $c$ of $G^*$ with values in $A$ and inducing $f$ can
be constructed by first choosing an arbitrary vertex $v^*$ of the dual graph
and assigning it an arbitrary value in $A$, and then propagating the
colors to the remaining vertices so that the coloring satisfies \eqref{eq:finduc}. As a result, each nowhere-zero $A$-flow in $G$ is
induced by exactly $|A|$ distinct proper colorings of $G^*$ with
values in $A$. 

We now observe that adjacent colorings in $G^*$ induce adjacent flows
in $G$. We recall that $\mathcal{C}(G,k)$ is the reconfiguration graph of
all proper $k$-colorings of $G$, where two colorings are adjacent if
and only if they differ on a single vertex. 

\begin{lemma}\label{color_to_flow}
    Let $G$ be a 2-edge-connected plane graph, and let $A$ be an
    abelian group. Let $c$ and $c'$ be proper
    colorings of $G^*$ with colors in $A$, such that $c$ and $c'$ are
    adjacent in $\mathcal{C} (G^*,|A|)$. Then the nowhere-zero
    $A$-flows $f$ and $f'$, induced by $c$ and $c'$ respectively, are
    adjacent in $\mathcal{F}(G,A)$. Moreover, if $c$ and $c'$ have
    colors in $\{1,\ldots,k\}$, then the nowhere-zero $k$-flows $f$
    and $f'$ are adjacent in $\mathcal{F}(G,k)$.
\end{lemma}
    
\begin{proof}
  Let $v^*$ be the unique vertex of $G^*$ on which $c$ and $c'$ differ, and
  let $a=c'(v^*)-c(v^*)\ne 0$. Let $C$ be the cycle of $G$ bounding
  the face of $G$ corresponding to the vertex $v^*$ of $G^*$. Assume
  without loss of generality that $C$ is directed clockwise in the
  underlying orientation of $G$. It follows that all edges of $G^*$
  incident to $v^*$ are oriented towards $v^*$. Let $e$ be an edge of
  $C$. Then $f(e)=c(v^*)-c(u^*)$, where $u^*$ is the endpoint of the
  dual edge $e^*$ distinct from $v^*$. We
  have \[f'(e)=c'(v^*)-c'(u^*)=c'(v^*)-c(v^*)+c(v^*)-c(u^*)=a+f(e).\]
  For every edge $e$ not in $C$, $e^*$ is not incident to $v^*$ and
  thus $f'(e)=f(e)$. It follows that $\supp(f'-f)=C$ and thus $f$ and
  $f'$ are adjacent in $\mathcal{F}(G,A)$ (and in $\mathcal{F}(G,k)$
  if $c$ and $c'$ have colors in $\{1,\ldots,k\}$).
\end{proof}

We note that the converse of Lemma \ref{color_to_flow} does not hold
in general. To see this, observe first that $K_4$ is self-dual. We have seen that
$\mathcal{F}(K_4,\mathbb{Z}_4)$ and $\mathcal{F}(K_4,4)$ contain
some edges and $\mathcal{F}(K_4,\mathbb{Z}_2\times \mathbb{Z}_2)$ is
even connected. However, $\mathcal{C} (K_4,4)$ is an edgeless graph on $4!$
vertices.

\medskip

We now use Lemma \ref{color_to_flow} to translate results on recoloring in
planar graphs to results on group flow reconfiguration.

\begin{theorem}\label{thm:planar}
Let $G$ be a 2-edge-connected plane graph and let $A$ be an abelian
group. Then the diameter of  $\mathcal{F}(G,A)$ is at most the
diameter of $\mathcal{C}(G^*,|A|)$. In particular,  if
$\mathcal{C} (G^*,|A|)$ is connected, then $\mathcal{F}(G,A)$
is also connected.
\end{theorem}

\begin{proof}
Let $f$ and $f'$ be two nowhere-zero $A$-flows in $G$, and let $c$ and $c'$ be proper
    colorings of $G^*$ with values in $A$ inducing $f$ and $f'$,
    respectively. Then there is a path $c_1,\ldots, c_s$ between
    $c_1=c$ and $c_s=c'$ in $\mathcal{C}(G^*,|A|)$, of length
    at most the diameter of $\mathcal{C}(G^*,|A|)$. For any
    $1\le i \le s$, let $f_i$ be the nowhere-zero $A$-flow induced by
    $c_i$. By Lemma \ref{color_to_flow}, any two flows $f_i,f_{i+1}$
    are adjacent in $\mathcal{F}(G,A)$, and thus the distance between
    $f=f_1$ and $f'=f_s$ is at most the diameter of $\mathcal{C}(G^*,|A|)$.
  \end{proof}

We now extend Theorem
\ref{thm:planar} to integer flows. Recall that for any integer $k$ and
plane graph $G$, any proper $k$-coloring of $G^*$ induces a
nowhere-zero $k$-flow in $G$. Conversely, consider a nowhere-zero
$k$-flow $f$ in $G$, and let $\bar{f}$ be the nowhere-zero
$\mathbb{Z}_k$-flow of $G$ obtained from $f$ by taking all flow values
modulo $k$. Then there is a proper coloring $\bar{c}$ of $G^*$
with colors in $\mathbb{Z}_k$ which induces $f$. Let $c$ be the proper
$k$-coloring of $G^*$ obtained by replacing every
element of $\mathbb{Z}_k$ by its unique representative in
$\{1,\ldots,k\}$. We say that $c$ \emph{is
associated to $f$}. It would be tempting to think that if $c$ is
associated to $f$, then $c$  induces $f$,
but this is not necessarily the case. Instead, we have the following
weaker property.

\begin{observation}\label{obs:modduality}
Let $f$ be a nowhere-zero $k$-flow in a plane graph $G$ and let $c$ be
a proper $k$-coloring of $G^*$ associated to $f$. Let $f'$ be the
nowhere-zero $k$-flow of $G$ induced by $c$. Then $(f'-f)(e)\equiv 0\pmod k$ for any edge $e$ of $G$.
\end{observation}

\begin{proof}
  Let $\bar{f}$ and $\bar{c}$ be as defined in the paragraph above.
Let $e$ be an edge of $G$ and let $e^*=(u^*,v^*)$ be the dual edge of $e$ in
$G^*$. Then \[{f'(e) \pmod k}\equiv {c(v^*)-c(u^*)\pmod
    k}\equiv\bar{c}(v^*)-\bar{c}(u^*)\equiv \bar{f}(e),\] and thus
$f'(e)-f(e)\equiv 0\pmod k$, as desired.
\end{proof}

Next, we will need the following simple result, showing that two flows $f$
and $f'$ as in Observation \ref{obs:modduality} above are connected by
a short path in $\mathcal{F}(G,k)$. 

\begin{lemma}\label{lem:int0}
Let $G$ be a graph and let $k\ge 2$ be an integer. Let $f,f'$ be two
nowhere-zero $k$-flows in $G$, such that $(f'-f)(e)\equiv 0\pmod k$ for any edge $e$ of $G$.  Then $f$ and $f'$ are at distance at most
$|E(G)|$ apart in
$\mathcal{F}(G,k)$. 
\end{lemma}

\begin{proof}
Observe first that as $f$ and $f'$ are $k$-flows, for every edge $e$ of
$G$, we have $|(f-f')(e)|<2k$. Since $(f-f')(e)\equiv 0\pmod k$, this
implies $(f-f')(e)\in \{-k,0,k\}$. Up to reversing 
a subset of the edges of $G$ in the underlying orientation of $G$, we can assume
without loss of generality that $(f-f')(e)\in \{0,k\}$ for any edge $e$
of $G$, and thus the support of $f-f'$ is an Eulerian digraph. It
follows that the support of $f-f'$ can be decomposed into edge-disjoint directed
cycles, each with flow value $k$. Adding each of these cycles (with
flow value $k$) sequentially to $f$, we obtain a path from $f$
to $f'$ of length at most $|E(G)|$ in $\mathcal{F}(G,k)$, as desired. 
\end{proof}

We are now ready to prove the following version of Theorem
\ref{thm:planar} for integer flows.

\begin{theorem}\label{thm:planarint}
Let $G$ be a 2-edge-connected plane graph and $k\ge 2$ be an
integer. Then the diameter of  $\mathcal{F}(G,k)$ is at most $2|E(G)|$
plus the
diameter of $\mathcal{C}(G^*,k)$. In particular,  if
$\mathcal{C} (G^*,k)$ is connected, then $\mathcal{F}(G,k)$
is also connected.
\end{theorem}

\begin{proof}
  Let $f$ and $g$ be two nowhere-zero $k$-flows in $G$, and let $c$
  and $d$ be two proper $k$-colorings of $G^*$ associated to $f$ and
  $g$, respectively. Then there is a path $c_1,\ldots,c_s$ in
  $\mathcal{C} (G^*,k)$ between $c=c_1$ and $d=c_s$ of length
  at most the diameter of $\mathcal{C} (G^*,k)$. For any $1\le
  i \le s$, let $f_i$ be the nowhere-zero $k$-flow of $G$ induced by
  $c_i$. By Lemma \ref{color_to_flow}, there is an edge between $f_i$
  and $f_{i+1}$ in $\mathcal{F}(G,k)$ for any $1\le i < s$. By
  Observation \ref{obs:modduality} and Lemma \ref{lem:int0}, $f$ and
  $f_1$ lie at distance at most $|E(G)|$ in $\mathcal{F}(G,k)$ and
  similarly $g$ and
  $f_s$ lie at distance at most $|E(G)|$ in $\mathcal{F}(G,k)$. It
  follows that $f$ and $g$ lie at distance at most $2|E(G)|$ plus the
  diameter of $\mathcal{C}(G^*,k)$ in $\mathcal{F}(G,k)$.
\end{proof}

For our first application, consider the graph $D_m$ on two vertices
$u$ and $v$, with
$m\ge 2$ multiple
edges connecting $u$ and $v$. It can be checked that $\mathcal{F}(D_m,\mathbb{Z}_3)$ is
connected if and only if $m=2$ or $m=4$. On the other hand, the dual
graph of $D_m$ is the cycle $C_m$, which is 2-degenerate, thus $\mathcal{C}(C_m, k)$ is connected whenever $k\geq 4$ (using Theorem \ref{coloring_degeneracy} or by a direct proof, the latter providing a $O(m)$ upper bound on the diameter of $\mathcal{C}(C_m, k)$). We
directly obtain the following (for which we have not been able to found a proof
avoiding planar duality).

 \begin{corollary}
 For any
integers $m\ge 2$ and  $k\ge 4$ and any abelian group $A$ of cardinality at least
$4$, each of the reconfiguration graphs
$\mathcal{F}(D_m,k)$ and $\mathcal{F}(D_m,A)$  is connected.
\end{corollary}

Note that multiple edges are irrelevant in graph coloring, and thus $\mathcal{C}(G^*,k)$ does not depend on the number
of multiple edges between pairs of vertices of $G^*$, so we can
consider the simple graph underlying $G^*$ instead of $G^*$. (Equivalently, this
corresponds to contracting edges in $G$ to make it 3-edge-connected). 
By Theorem \ref{coloring_degeneracy}, we obtain the following direct corollary of Theorems \ref{thm:planar}
and \ref{thm:planarint}.

    \begin{corollary}
Let $G$ be a 2-edge-connected plane graph such that the simple planar
graph underlying the dual graph
$G^*$ is $d$-degenerate, for some integer $d\ge 1$. Then for any
integer $k\ge d+2$ and any abelian group $A$ of cardinality at least
$d+2$, each of the reconfiguration graphs
$\mathcal{F}(G,k)$ and $\mathcal{F}(G,A)$  is connected.
\end{corollary}

Simple planar graphs are 5-degenerate, and more generally it follows from
Euler's formula that for any $g\ge 3$, simple planar graphs of  girth (length
of a shortest cycle) at
least $g$ have average degree less than $2g/(g-2)$. In particular,
simple planar graphs of girth at least 4 are 3-degenerate and simple planar graphs
of girth at least 6 are 2-degenerate. Using the fact that the edge-connectivity of a planar graph is the same as the girth of its dual, this directly
implies the following.

   \begin{corollary}
     Let  $G$ be a 2-edge-connected planar graph. For any
integer $k\ge 7$ and any abelian group $A$ of cardinality at least
$7$, each of the reconfiguration graphs
$\mathcal{F}(G,k)$ and $\mathcal{F}(G,A)$  is connected.
\end{corollary}

  \begin{corollary}
     Let  $G$ be a 4-edge-connected planar graph. For any
integer $k\ge 5$ and any abelian group $A$ of cardinality at least
$5$, each of the reconfiguration graphs
$\mathcal{F}(G,k)$ and $\mathcal{F}(G,A)$  is connected.
\end{corollary}

  \begin{corollary}[using \protect{\cite[Theorem 1.3]{Girth_five}}]\label{cor:g5}
     Let  $G$ be a 5-edge-connected planar graph. For any
integer $k\ge 4$ and any abelian group $A$ of cardinality at least
$4$, each of the reconfiguration graphs
$\mathcal{F}(G,k)$ and $\mathcal{F}(G,A)$  is connected.
\end{corollary}

Note that the simple graph underlying the dual graph of any
outerplanar graph is 2-degenerate. We thus obtain the following.

  \begin{corollary}
     Let  $G$ be an outerplanar graph. For any
integer $k\ge 4$ and any abelian group $A$ of cardinality at least
$4$, each of the reconfiguration graphs
$\mathcal{F}(G,k)$ and $\mathcal{F}(G,A)$  is connected.
\end{corollary}

Note that the proof of Theorem \ref{coloring_degeneracy} does not
provide  good upper bounds on the diameter of the reconfiguration
graphs in general. However, such bounds typically exist in the planar case (see for instance \cite[Table 1]{Girth_five} and
the reference therein), and thus Theorems \ref{thm:planar}
and \ref{thm:planarint} automatically provide good upper bounds on the
diameter of $\mathcal{F}(G,A)$ for the dual cases. This gives polynomial bounds on the diameter of the flow reconfiguration graph in all the results in this section, with the notable exception of Corollary \ref{cor:g5}.

\section{Group flow reconfiguration in general graphs}\label{sec:ab}

This section is devoted to the reconfiguration of group flows in general graphs. In Section \ref{sec:groupflowprel}, we introduce a useful tool that will allow us to extract short subdivisions of highly connected subgraphs in any sufficiently dense graph. This will be used repeatedly in the proofs of our results. In Section \ref{sec:choos}, we present a number of old and new results on flow choosability, which will be crucial in the remainder of this section. In Section \ref{sec:smallproduct} we prove that $\mathcal{F}(G,\mathbb{Z}_2^{8})$ is connected for every 2-edge-connected graph $G$.  In Section \ref{sec:cyclic} we prove that if $A$ is any sufficiently large cyclic group, $\mathcal{F}(G,A)$ is connected for every 2-edge-connected graph $G$. Finally, in Section \ref{sec:largeabelian}, we combine these results to show that for any sufficiently large abelian group $A$, $\mathcal{F}(G,A)$ is connected for  every 2-edge-connected graph $G$.

\subsection{Preliminaries}\label{sec:groupflowprel}

Given a (multi)graph $G$, a \emph{$(\le k)$-subdivision} of $G$ is a graph
obtained from $G$ by replacing every edge $uv$ by a path on at most $k+2$
vertices between $u$ and $v$.

\smallskip

We will need the following simple lemma.

\begin{lemma}\label{lem:kevin}
Let $G$ be a graph with $\ve(G)\ge 3$ and $\e(G)> \frac{4}{3}\, \ve(G)-2$. Then $G$ contains a
subgraph which is a cycle of length at most 5 or a $(\le 2)$-subdivision of a
3-edge-connected graph.
\end{lemma}

\begin{proof}
  We prove the result by induction on $\ve(G)\ge 3$. Note that since
  $\ve(G)\ge 3$, we have $\e(G)> \frac{4}{3}\, \ve(G)-2\ge \ve(G)-1$, and thus
  $G$ contains a cycle. Hence, we can assume that $\ve(G)\ge 6$ and $G$ has no multiple edges
  (otherwise $G$ clearly contains a cycle of length at most $5$). Since for every $k\ge 6$, $\e(C_k)=k \leq \frac{4}{3}\,k-2=\frac{4}{3}\,\ve(C_k)-2$, it
  follows that $G$ itself is not a cycle.

  \smallskip

  Note that $G$ has no vertex $v$ of degree at most 1, since
  otherwise \[\e(G-v)=e(G)-1> \tfrac{4}{3}\, \ve(G)-3> \tfrac{4}{3}\, \ve(G-v)-2,\]
  and since $\ve(G')\ge 6-1=5$ we can apply induction on $G'$. So we
  can assume that $G$ has minimum degree 2.

  \smallskip

  This shows that all components have size at least
  3. By an averaging argument, at least one component $C$ must
  satisfy $\e(C)> \tfrac43\, \ve(C)-2$, and we can then apply induction
  to $C$ if $C\neq G$. Hence, we can assume that $G$ is connected.

  \smallskip


Assume that $G$  is not $3$-edge-connected. Consider any
partition  of the vertex set $V$ of $G$ into two non-empty sets
$V_1,V_2$ such that there are
at most two edges with one endpoint in $V_1$ and the other in $V_2$ in
$G$.  Let $G_1=G[V_1]$ and $G_2=G[V_2]$, and assume by symmetry that $\ve(G_1)\le \ve(G_2)$. If $\e(G_1)\le \tfrac{4}{3}\,\ve(G_1)-2$
and $\e(G_2)\le \tfrac{4}{3}\,\ve(G_2)-2$ we have \[\e(G)\le \e(G_1)+\e(G_2)+2\le \tfrac{4}{3}
  \,\ve(G)-4+2=\tfrac{4}{3} \,\ve(G)-2,\] a contradiction. So there exists
$i\in \{1,2\}$ such that  $\e(G_i)> \tfrac{4}{3}\,
\ve(G_i)-2$. If $\ve(G_i)\ge 3$ we can apply
induction to $G_i$ and obtain a subgraph which is a cycle of length at most $5$ or a $(\le 2)$-subdivision of a
3-edge-connected graph. So we can assume that $\ve(G_1)\le \ve(G_i)\le
2$. Note that in this case $G_1$ consists of a single vertex or a
single edge. In particular the edge-cut separating $V_1$ and $V_2$
contains precisely two edges, since otherwise $V_1$ would contain a
vertex of degree $1$ in $G$, which is a contradiction. So $G$ is $2$-edge-connected.

A \emph{thread} in $G$ is a path where all internal vertices have
degree $2$ in $G$ and the endpoints have degree more than $2$. By considering the 2-edge-cut consisting of the first and last edge of a thread, the paragraph above shows that the thread must contain of at most four vertices. This implies that
$G$ does not contain any thread of at least five vertices. Replace
every thread in $G$ by a single edge (possibly producing multiple edges). Recall that $G$ has minimum degree at least $2$ and is not
a cycle, so the operation above is well defined and indeed produces
some graph $H$, and $G$ is a $(\le 2)$-subdivision of $H$ by
definition. As every $2$-edge-cut in $H$ can be extended to a $2$-edge-cut
of $G$ where both sides of the cut have size at least $3$, the graph  $H$
must be  $3$-edge-connected, as desired.
\end{proof}

\subsection{Flow choosability}\label{sec:choos}

We will need the following results on the existence of flows
avoiding certain flow values.

\smallskip

Given a abelian group $A$ and an integer $k$, we say that a graph $G$ with underlying orientation $D$ is \emph{$(A,k)$-flow-choosable} if for any list assignment $F:E(G)\to 2^A$, where $|F(e)|\le k$ for every $e\in E(G)$ (seen as  a set of \emph{forbidden} values),  there is an $A$-flow $f$ of $G$ with respect to $D$
  such that for every edge $e$ in $G$, $f(e)\not\in F(e)$. As before, we often omit the explicit mention of $D$ and consider each graph $G$ as having an underlying (fixed) orientation, and all flows of $G$ are defined with respect to this orientation. When we really want to stress that this orientation is relevant for the problem we consider, we simply say that $G$ is an oriented graph.

\begin{theorem}[\cite{DeV00}]\label{thm:matt}
For every integer $k\ge 4$, every 3-edge-connected oriented graph $G$  is $(\mathbb{Z}_2^{k},3)$-flow-choosable.
\end{theorem}

We will also need the following more general version (with slightly weaker bounds).
  
  \begin{theorem}\label{thm:gco}
   Let $A$ and $B$ be two abelian groups with $|A|\ge 7$ and $|B|\ge 4$. Then for every 3-edge-connected oriented graph $G$ and every  vertex $v$ of
degree 3 of G, $G - v$ is $(A\times B,3)$-flow-choosable. In particular, every 3-edge-connected graph is $(A\times B,3)$-flow-choosable. 
  \end{theorem}

We deduce Theorem \ref{thm:gco} from results from \cite{groupConnectivity} as we explain now.

\medskip

We say that a graph $G$ is the \emph{2-closure} of a spanning subgraph $H$ of $G$ if there is a sequence of cycles $C_1,\ldots,C_k$ in $G$ such that $E(G)=E(H)\cup E(C_1) \cup \cdots \cup E(C_k)$ and for every $i\ge 1$, $C_i$ contains at most 2 edges that are not in  $E(H)\cup E(C_1) \cup \cdots\cup E(C_{i-1})$.

\begin{lemma}[Lemma 2.1 in \cite{groupConnectivity}]\label{lem:21}
Let $G$ be an oriented graph and $T$ be a spanning tree of $G$ such that $G$ is the 2-closure of the spanning subgraph $R$ of $G$ with edge-set $E(G)-E(T)$. Let $A$ be an abelian group, and assume that every edge
  $e\in E(T)$ has a list $F(e)$ of less than $|A|/2$ values in
  $A$. Then there is an $A$-flow $f$ of $G$
  such that for every edge $e\in E(T)$, $f(e)\not\in F(e)$.
\end{lemma}

The proof of Lemma 3.1 in \cite{groupConnectivity} yields verbatim the following result.

\begin{lemma}\label{lem:31}
Theorem \ref{thm:gco} is equivalent to its restriction to cubic graphs.
\end{lemma}

\begin{proof}
 Observe that every connected graph is $(A\times B,3)$-flow-choosable if and only if all its 2-connected components are $(A\times B,3)$-flow-choosable. As a vertex of degree 3 cannot be a cut-vertex in a 3-edge-connected graph, we can thus assume that the graph $G$ in the statement of Theorem \ref{thm:gco} is 2-connected. Replace every vertex $u$ of degree at least 4 by a cycle of length $d_G(u)$ (connecting each edge incident to $u$ in $G$ to a different vertex of the cycle). The resulting graph $H$ is cubic and 3-edge-connected, and any vertex $v$ of degree 3 in $G$ is still a vertex of degree 3 in $H$. As $G-v$ can be obtained from $H-v$ by edge-contractions, if $H-v$ is $(A\times B,3)$-flow-choosable, then $G-v$ is also $(A\times B,3)$-flow-choosable. It follows that Theorem \ref{thm:gco} is equivalent to its reduction to cubic graphs.
\end{proof}

A graph is said to be a \emph{truncated cubic 3-edge-connected graph} if it is obtained from a cubic 3-edge-connected graph with at least 4 vertices by
the deletion of a single vertex.

\begin{lemma}[Lemma 3.3 in \cite{groupConnectivity}]\label{lem:33}
Every truncated cubic 3-edge-connected graph $G$ has a spanning tree $T$  such that $G$ is the 2-closure of the spanning subgraph $R$ of $G$ with edge-set $E(G)-E(T)$. 
\end{lemma}

We are now ready to prove Theorem \ref{thm:gco}, which is inspired by the proof of Theorem 3.2 in \cite{groupConnectivity}. 

\medskip

\noindent \emph{Proof of Theorem \ref{thm:gco}.}
By Lemma \ref{lem:31}, it is enough to prove that every truncated cubic 3-edge-connected graph $G$ is $(A\times B,3)$-flow-choosable. Fix such a graph $G$, and some lists $F(e)$ of forbidden flow values on the edges of $G$, with $|F(e)|\le 3$ for every edge $e\in E(G)$. By Lemma \ref{lem:33}, $G$ has a spanning tree $T$ such that $G$ is the 2-closure of the spanning subgraph $R$ of $G$ with edge-set $E(G)-E(T)$. For every edge $e\in E(G)$, define $F'(e)=\{z+(0_A,x),|\,z\in F(e), x\in B\}$. Note that since $|F(e)|\le 3$ and $|B|\ge 4$, we have $|F'(e)|\le 3|B|<|A|\cdot |B|/2=|A\times B|/2$ (where the second inequality follows from $|A|\ge 7)$.

We can thus apply Lemma \ref{lem:21}, and obtain  an $(A\times B)$-flow $f$ in $G$ such that for every $e\in E(T)$ and any $z\in F(e)$, $f(e)-z$ is not in the subgroup $\{0_A\}\times B$. For any edge $e\not\in E(T)$, choose some value $x_e$ in the subgroup $\{0_A\}\times B$, such that $f(e)+x_e\not\in F(e)$ (this is always possible, as $|F(e)|\le 3$ and there are $|B|\ge 4$ choices for $x_e$). Starting from $f$, add for each edge  $e\in E(G)-E(T)$ some flow of value $x_e$ along the unique cycle consisting of $e$ and edges of $T$. By the definition of $x_e$, the resulting flow (call it $g$), satisfies $g(e)=f(e)+x_e\not\in F(e)$ for each edge $e\in E(G)-E(T)$. Moreover, for every edge $e\in T$, all flow values of $g-f$ are in $\{0_A\}\times B$. As for every $e\in E(T)$ and any $z\in F(e)$, $f(e)-z$ is not in the subgroup $\{0_A\}\times B$, it follows that $g(e)-z$ is not in the subgroup $\{0_A\}\times B$, and thus $g(e)\ne z$ for any $z \in F(e)$, as desired. \hfill $\Box$

\medskip

We obtain the following immediate corollary of Theorem \ref{thm:matt}.

\begin{corollary}\label{cor:raphael}
  Let $\ell\ge 4$. Let $G$ be an orientation of a cycle of length at most 5 or of a $(\le
  2)$-subdivision of a 3-edge-connected graph and let $\phi: E(G)\to  \mathbb{Z}_2^{\ell}$ be some function. Then there is a $\mathbb{Z}_2^{\ell}$-flow $f$ of $G$
  such that for every edge $e$ in $G$, $f(e)\neq \phi(e)$.
\end{corollary}


\begin{proof}
 Assume first that $G$ is a cycle of length at most 5. Then each choice
of a flow value along some edge $e$ forces all the other flow values. As
there are at most 5 edges and at most 1 forbidden value per edge,
there are at most $5$ forbidden values for $e$. Since
$|\mathbb{Z}_2^{\ell}|=2^{\ell}>5$ we can always find a
$\mathbb{Z}_2^{\ell}$-flow where every edge $e$ is assigned a
flow value distinct from $\phi(e)$.

Assume now that $G$ is a $(\le
  2)$-subdivision of a 3-edge-connected graph $H$. As every
  subdivided edge of $G$ corresponds to a path  containing at most 3
  edges in $H$, a flow of $G$ avoiding at most 1 values per edge
  correspond to flows in $H$ avoiding at most $3$ values per edge. By 
  Theorem \ref{thm:matt}, it follows that $G$ has a $\mathbb{Z}_2^{\ell}$-flow where every edge $e$ is assigned a
flow value distinct from $\phi(e)$.
\end{proof}

%

The same proof, using Theorem \ref{thm:gco} instead of Theorem \ref{thm:matt}, gives the following analogous result.

\begin{corollary}\label{cor:gcoAB}
  Let $A$ and $B$ be two abelian groups with $|A|\ge 7$ and $|B|\ge 4$. Let $G$ be an orientation of a cycle of length at most 5 or of a $(\le
  2)$-subdivision of a 3-edge-connected graph and let $\phi: E(G)\to  A\times B$ be some function. Then there is a $(A\times B)$-flow $f$ of $G$
  such that for every edge $e$ in $G$, $f(e)\neq \phi(e)$.
\end{corollary}

%

Since 4 and 7 are coprime, $\mathbb{Z}_{28}$ is isomorphic to $\mathbb{Z}_{7}\times \mathbb{Z}_4$. We obtain the following immediate corollary.

 \begin{corollary}\label{cor:gco}
Let $G$ be an orientation of a cycle of length at most 5 or of a $(\le
  2)$-subdivision of a 3-edge-connected graph and let $\phi: E(G)\to  \mathbb{Z}_{28}$ be some function. Then there is a $\mathbb{Z}_{28}$-flow $f$ of $G$
  such that for every edge $e$ in $G$, $f(e)\neq \phi(e)$.
 \end{corollary}

By Theorem \ref{thm:tutteintmod}, we deduce the following.

\begin{corollary}\label{cor:28flow}
   Let $G$ be an orientation of a cycle of length at most 5 or of a $(\le
  2)$-subdivision of a 3-edge-connected graph and let $\phi: E(G)\to  [-27,27]$ be some function. Then there is a $28$-flow $f$ of $G$
  such that for every edge $e$ in $G$, $f(e)\neq \phi(e)$.
\end{corollary}


\subsection{Direct product of small groups}\label{sec:smallproduct}

In the remainder of this section it is convenient to view $\mathbb{Z}_2^{k}$ as a vector space over $\mathrm{GF}(2)$, and its subgroups as linear subspaces.
We recall that when some vector space $V$ is the direct sum $V=A \oplus B$ of two linear subspaces $A$ and $B$, any element $x\in V$ can be written uniquely as $x=x_A+x_B$, with $x_A\in A$ and $x_B\in B$. In this case, we say that $A$ and $B$ are \emph{complementary} and we call $x_A$ (resp.\ $x_B$) the \emph{projection} of $x$ to $A$ (resp.\ $B$).
\smallskip

The following simple observation will be used repeatedly.

\begin{observation}\label{obs:sum}
Let $A$ and $B$ be  complementary linear subspaces of $\mathbb{Z}_2^{k}$. Let $G$ be a 2-edge-connected graph, and let $f_A$ be a nowhere-zero $A$-flow in  $G$. Then for any $B$-flows $g_{B}$ and $h_{B}$ in $G$, $f_A+g_{B}$ is in the same connected component as $f_A+h_{B}$ in $\mathcal{F}(G,\mathbb{Z}_2^{k})$.
\end{observation}

\begin{proof}
Using Lemma \ref{lem:linear}, decompose the flow $h_{B}-g_{B}$ as a sum of $B$-flows whose supports are cycles of $G$. Starting with $f_A+g_{B}$, add each of these flows one after the other. The resulting flow is $f_A+h_{B}$, and since $f_A$ is non-zero, all intermediate flows are nowhere-zero $\mathbb{Z}_2^{k}$-flows.
\end{proof}

We will later also use the following variant of Observation \ref{obs:sum}. 

\begin{observation}\label{obs:sum2}
Let $A$ and $B$ be  abelian groups. Let $G$ be a 2-edge-connected graph, and let $f_A$ be a nowhere-zero $A$-flow in  $G$. Then for any $B$-flows $g_{B}$ and $h_{B}$ in $G$, $(f_A,g_{B})$ is in the same connected component as $(f_A,h_{B})$ in $\mathcal{F}(G,A\times B)$.
\end{observation}

Let us say that a nowhere-zero $\mathbb{Z}_2^{k}$-flow $f$ in a graph $G$ is \emph{composite} if there exist two complementary linear subspaces $A$ and $B$ of $\mathbb{Z}_2^{k}$ such that the projections of $f$ to $A$ and $B$ are both nowhere-zero, by which we mean that every flow value $f(e)$ can be written uniquely as $f(e)=f_A(e)+f_{B}(e)$, where $f_A(e)\in A\setminus \{0\}$ and $f_{B}(e)\in B\setminus \{0\}$.

\begin{lemma}\label{lem:composite}
For $k\ge 8$, all composite nowhere-zero $\mathbb{Z}_2^{k}$-flows of a 2-edge-connected graph $G$ are in the same component of $\mathcal{F}(G,\mathbb{Z}_2^{k})$.
\end{lemma}

\begin{proof}
Consider two composite nowhere-zero $\mathbb{Z}_2^{k}$-flows $f$ and $g$ of $G$. 
Let $A_1$ and $A_2$ be complementary linear subspaces of $\mathbb{Z}_2^k$ such that  the projections of $f$ to both $A_1$ and $A_2$ are nowhere-zero.
Likewise, let $B_1$ and $B_2$ be complementary linear subspaces of $\mathbb{Z}_2^k$ such that  the projections of $g$ to both $B_1$ and $B_2$ are nowhere-zero.
Without loss of generality, both $A_1$ and $B_1$ have dimension at most $4$. Let $A_2'$ be a 3-dimensional linear subset of $A_2$ and $B_2'$ be a 3-dimensional linear subset of $B_2$.

Let $f'$ be a nowhere-zero $A_2'$-flow of $G$, and let $g'$ be a nowhere-zero $B_2'$-flow of $G$ (these flows exist, as $A_2'$ and $B_2'$ have dimension exactly 3 and every 2-edge-connected graph has a nowhere-zero $\mathbb{Z}_2^3$-flow).
By Observation \ref{obs:sum}, $f$ is in the same connected component of $\mathcal{F}(G,\mathbb{Z}_2^{k})$ as $f_{A_1}+f'$. Applying Observation \ref{obs:sum} again, this latter flow is in the same connected component as $f'$. So we obtain that $f$ and $f'$ are in the same component of $\mathcal{F}(G,\mathbb{Z}_2^{k})$, and similarly that $g$ and $g'$ are in the same component of $\mathcal{F}(G,\mathbb{Z}_2^{k})$.

\smallskip

It remains to show that $f'$ and $g'$ are in the same component of $\mathcal{F}(G,\mathbb{Z}_2^{k})$.
If $A_2'\cap B_2'=\{0\}$, they both are in the same component as $f'+g'$, as an immediate consequence of Observation \ref{obs:sum} (applied twice, as above).
Thus, we may assume that $\dim(A_2'+ B_2')\le 5$.
Let $C$ be a 3-dimensional subspace of $\mathbb{Z}_2^k$ with $(A'_2+B_2')\cap C=\{0\}$, and let $h$ be a nowhere-zero $C$-flow in $G$.
As above, Observation \ref{obs:sum} (applied twice) shows that each of $f'$ and $g'$ is in the same connected component of $\mathcal{F}(G,\mathbb{Z}_2^{k})$ as $h$, which completes the proof.
\end{proof}

The next result is the final ingredient we need in order to 
prove the main result of this section.

\begin{lemma}\label{lem:gentocompo}
Let $G$ be a 2-edge-connected graph and $k\ge 8$ be an integer. For any nowhere-zero $\mathbb{Z}_2^k$-flow $f$ in $G$, there exists a composite nowhere-zero $\mathbb{Z}_2^k$-flow $g$ in $G$ such that $f$ and $g$ are in the same connected component of $\mathcal{F}(G,\mathbb{Z}_2^{k})$.
\end{lemma}

\begin{proof}
  Let $G$ be a 2-edge-connected graph. We proceed by induction on
  $\ve(G)$. Assume first that $G$ has a
  2-edge-cut $\{e_1,e_2\}$. Let $f'$ be the restriction of $f$ to $G/e_1$, the (multi)graph obtained from $G$ by identifying the two
  endpoints of $e_1$ into a single vertex. By induction, $f'$ lies in the same connected component as some composite nowhere-zero $\mathbb{Z}_2^{k}$-flow $g'$ of $G/e_1$ in $\mathcal{F}(G/e_1,\mathbb{Z}_2^{k})$. It can then be checked that the $\mathbb{Z}_2^{k}$-flow $g$ of $G$ obtained from $g'$ by uncontracting $e_1$ and assigning it flow value $-g(e_2)$ is nowhere-zero, composite and lies in the same connected component of $\mathcal{F}(G,\mathbb{Z}_2^{k})$ as $f$, as desired.
   Hence, we can assume
  without loss of generality that $G$ is 3-edge-connected, and in
  particular, $\e(G)\ge \tfrac32\, \ve(G)$.

  Consider a  linear subspace $B$ of $\mathbb{Z}_2^{k}$, taken uniformly at random among all 4-dimensional linear subspaces. Let $A$ be a linear subspace such that $\mathbb{Z}_2^{k}=A\oplus B$. Since $k\ge 8$, $A$ and $B$ have dimension at least 4. We can write uniquely every element $x\in \mathbb{Z}_2^{k}$  as $x_A+x_B$, with $x_A\in A$ and $x_B\in B$. By extension, we write every $\mathbb{Z}_2^{k}$-flow $h$ of $G$ as $h=h_A+h_B$, where $h_A$ is an $A$-flow of $G$ and $h_B$ is a $B$-flow of $G$.
  
 For every non-zero element $x\in \mathbb{Z}_2^{k}$, the probability that $x_A$ is zero is the probability that $x\in B$, which is $\tfrac{2^4-1}{2^k-1}\le 1/17$. It follows that in expectation, the number of edges $e$ of $G$ for which $f_A(e)$ is non-zero is at least \[\tfrac{16}{17}\cdot \e(G)\ge \tfrac{16}{17}\cdot \tfrac32\, \ve(G)> \tfrac43\, \ve(G).\] From now on, we fix a pair of complementary linear subspaces $A,B$ of dimension at least 4, such that the number of edges $e$ of $G$ for which $f_A(e)$ is non-zero is greater than $\tfrac43\, \ve(G)$.

\smallskip

 Let $G'$ be the subgraph of $G$ spanned by the edges that are non-zero
 in $f_A$. By the paragraph above, $\e(G')> \tfrac43\, \ve(G)\ge \tfrac43\,
 \ve(G')$. In particular, either $\ve(G')=2$ and $\e(G')\ge 2$, or  $\ve(G')\ge 3$ and $\e(G')>
 \tfrac43\,\ve(G')-2$. It follows from Lemma \ref{lem:kevin} that in both cases $G'$ (and thus $G$) contains a subgraph $H$
 which is a cycle of length at most 5 or a $(\le 2)$-subdivision of a
 3-edge-connected graph. By definition, $f_A$ is
 nowhere-zero in $H$.

 We contract $H$ into a single vertex $v^*$ in $G$. The
 resulting graph is denoted by $G^*$, and the restriction of the nowhere-zero flow $f$ to $G^*$ is noted by $f^*$. By induction, there exists a composite nowhere-zero $\mathbb{Z}_2^{k}$-flow $g^*$ of $G^*$ in the same component of $f^*$ in $\mathcal{F}(G^*,\mathbb{Z}_2^{k})$. Let $h_A$ be a nowhere-zero $A$-flow in $G$ and $h_B$ be a nowhere-zero $B$-flow in $G$ (note that $h_A$ and $h_B$ exist, as $A$ and $B$ have dimension at least 4 and any 2-edge-connected graph has a nowhere-zero $\mathbb{Z}_2^\ell$-flow for any $\ell\ge 3$). Let $h^*$ be the restriction of $h=h_A+h_B$ to $G^*$. Note that $h$ and $h^*$ are composite nowhere-zero flows, by definition. Our goal in the remainder of the proof is to show that $f$ and $h$ are in the same component of $\mathcal{F}(G,\mathbb{Z}_2^{k})$. Since $k\ge 8$, it follows from Lemma \ref{lem:composite} that there is a path between $g^*$ and $h^*$ in  $\mathcal{F}(G^*,\mathbb{Z}_2^{k})$, and thus $f^*$ and $h^*$ are also in the same component of $\mathcal{F}(G^*,\mathbb{Z}_2^{k})$. Hence, there exists a
 sequence of values $(\lambda_i)_{1\le i \le \ell}$ from
 $\mathbb{Z}_2^{k}$  and cycles $\mathcal{C}^*=(C_i^*)_{1\le i \le \ell}$ of $G^*$
 such that $h^*$ can be obtained from $f^*$ by successively adding
 flow value $\lambda_i$ along cycle $C^*_i$ for $i=1,\ldots,\ell$, in
 such a way that all intermediate flows are nowhere-zero in $G^*$.

 Note that for each cycle $C_i^*$ in the sequence $\mathcal{C}^*$, there exists a
 cycle $C_i$ in $G$ that coincides with $C_i^*$ on $E(G)\setminus E(H)$. Starting with flow $f$, our
 goal will be to successively add
 flow value $\lambda_i$ along $C_i$ in $G$ for
 $i=1,\ldots,\ell$. By induction this will guarantee that all edges from $E(G)\setminus E(H)$
 remain non-zero during the process, but we need some additional preparation
 in $H$ before each step to make sure that edges of
 $H$ remain non-zero as well. We will maintain the property that
 just  after we add some flow value $\lambda_i$ along $C_i$, the
 current flow $q$ is such that $q_A$ is nowhere-zero in $H$. This
 property holds initially with $q=f$, by definition of $H$. 

Assume now that we are somewhere in the process, with some current nowhere-zero flow
$q=q_A+q_B$ such that $q_A$ is nowhere-zero in $H$. Our immediate goal is to
add some flow value $\lambda=\lambda_A+\lambda_B$ to some cycle $C$
of $G$, and we know (by induction) that all edges of $E(G)\setminus E(H)$ will remain non-zero
after this addition, so we only need to focus on the edges of
$C\cap E(H)$. By Corollary  \ref{cor:raphael},  there
exists a $B$-flow $r_B$ in $H$ such that for
each edge $e\in E(H)$, $r_B(e)\neq -q_B(e)$, so that
each edge $e$ satisfies $q_B(e)+r_B(e)\ne 0$. By Lemma \ref{lem:linear}, the flow $r_B$ can be decomposed into a sum of flows
whose supports are
cycles of $H$, so there is a path in $\mathcal{F}(G,
\mathbb{Z}_2^{k}) $ between flows $q$ and $q+r_B$ where all
intermediate flows $q'=q_A'+q_B'$ satisfy $q_A'=q_A$, and in particular
$q_A'$ is nowhere-zero in $E(H)$. In particular, all intermediate
flows between $q$ and $q+r_B$ are nowhere-zero. So, we have now reached the nowhere-zero $\mathbb{Z}_2^k$-flow $q+r_B$, whose main property is that both $q_A$ and $q_B+r_B$ are nowhere-zero in $H$.

By Corollary  \ref{cor:raphael}, there exists an $A$-flow $s_A$ in $H$ such that 
\begin{enumerate}[(i)]
    \item for
each edge $e\in E(H)\cap C$, $s_A(e)\neq -q_A(e)-\lambda_A$, and
\item for
each edge $e\in E(H)\setminus C$, $s_A(e)\neq-q_A(e)$.
\end{enumerate}
As before, the flow $s_A$ can be decomposed into a sum of flows
whose supports are
cycles of $H$, so there is a path in $\mathcal{F}(G,
\mathbb{Z}_2^{k}) $ between flows $q+r_B$ and $q+r_B+s_A$ where all
intermediate flows $q'=(q_A',q_B')$ satisfy $q_B'=q_B+r_B$, and in particular
$q_B'$ is nowhere-zero in $H$. In particular, all intermediate
flows between $q+r_B$ and $q+r_B+s_A$ are nowhere-zero. So, we have now reached the nowhere-zero $\mathbb{Z}_2^k$-flow $q+r_B+s_A$, whose main properties are that (1) $q_B+r_B$ is nowhere-zero in $H$, (2) for each edge $e\in E(H)\cap C$, $q_A(e)+s_A(e)\neq -\lambda_A$, and (3) for
each edge $e\in E(H)\setminus C$, $q_A(e)+s_A(e)\neq 0$.

 Once we have reached $q+r_B+s_A$, we can safely add $\lambda$
  along the cycle $C$: since the support of $r_B$ and $s_A$ is contained in $H$, all edges of $E(G)\setminus E(H)$ remain
  non-zero by induction, and by (2) and (3) above, the 
  nowhere-zero flow obtained after the addition of $\lambda$ along $C$,
  call it $q'$, is such that $q_A'$ is nowhere-zero in $E(H)$, which
  is the property we needed to maintain.

  We can now repeat the procedure above to add flow value
  $\lambda_i$ to each cycle $C_i$ in the sequence, until we reach some
  nowhere-zero flow $q$ which coincides with the composite flow $h$ on
  $E(G)\setminus E(H)$, and is such that $q_A$ is nowhere-zero in
  $E(H)$. As $h_A$ is nowhere-zero, this implies that $q_A$ is also nowhere-zero in $E(G) \setminus E(H)$, and thus in $G$. 
  By Observation \ref{obs:sum}, $q$ is in the same connected component as
  the nowhere-zero flow $q_A+h_B$ in $\mathcal{F}(G,\mathbb{Z}_2^k)$. The flow $h_B$ is also
  nowhere-zero in $G$, so by Observation \ref{obs:sum} again, $q_A+h_B$ is in the same connected component as $h=h_A+h_B$, as desired. 
\end{proof}

We can now easily deduce our main result.

\begin{theorem}\label{thm:groupflow}
For any 2-edge-connected graph $G$,
$\mathcal{F}(G,\mathbb{Z}_2^{8})$ is connected.
\end{theorem}

\begin{proof}
By Lemma \ref{lem:gentocompo}, every nowhere-zero $\mathbb{Z}_2^{8}$-flow in $G$ is in the same component of $\mathcal{F}(G,\mathbb{Z}_2^{8})$ as a composite nowhere-zero $\mathbb{Z}_2^{8}$-flow, and by Lemma \ref{lem:composite}, all such flows are in the same connected component. It follows that $\mathcal{F}(G,\mathbb{Z}_2^{8})$ is connected.
\end{proof}

In this section our aim was to minimize the size of a group $A$ such that $\mathcal{F}(G,A)$ is connected for every 2-edge-connected graph $G$, and for this it was somewhat more convenient to view $\mathbb{Z}_2^8$ as a vector space rather than a group. However the proof of Theorem \ref{thm:groupflow} can be carried more generally when the underlying group is the direct product of sufficiently many abelian groups. In Section \ref{sec:largeabelian}, we revisit the results of this section and explain how they can be extended to this more general setting (at the cost of significantly worse bounds). 

\medskip

Our aim in the remainder of this section is to show that for any sufficiently large abelian group $A$ and any 2-edge-connected graph $G$, $\mathcal{F}(G, A)$ is connected.

\subsection{Large cyclic subgroups}\label{sec:cyclic}

The goal of this section is to show that if a group $A$ contains a sufficiently large cyclic group, then for any 2-edge-connected graph $G$, $\mathcal{F}(G, A)$ is connected. For this we will need to do a surprising detour via integer flows (this is surprising in the sense that we understand integer flow reconfiguration much less than group flow reconfiguration, as witnessed by the much weaker results of Section \ref{sec:intflows} below).

\begin{lemma}\label{lem:cyclicgroups1}
Let $p$ and $q$ be coprime positive integers with $q\geq 55$ and $p>q^8\cdot 27+28$, and let $G$ be a $2$-edge-connected graph. Let $F:=\{-27,-26,\dots,26,27\}\subseteq \mathbb{Z}_p$.
For any nowhere-zero $\mathbb{Z}_p$-flow $f$ in $G$, there exists a nowhere-zero $\mathbb{Z}_p$-flow $g$ in $G$ such that $f$ and $g$ are in the same component of $\mathcal{F}(G,\mathbb{Z}_p)$ and $g(E(G))$ is disjoint from $F$.
\end{lemma}
\begin{proof}
Let $G$ be a 2-edge-connected graph. We proceed by induction on
  $\ve(G)$. Assume first that $G$ has a
  2-edge-cut $\{e_1,e_2\}$. Let $f'$ be the restriction of $f$ to $G/e_1$, the (multi)graph obtained from $G$ by identifying the two
  endpoints of $e_1$ into a single vertex. By induction, $f'$ lies in the same connected component as some  nowhere-zero $\mathbb{Z}_{p}$-flow $g'$ of $G/e_1$ in $\mathcal{F}(G/e_1,\mathbb{Z}_{p})$ such that $g'(E(G/e_1))$ is disjoint from $F$. It can then be checked that the $\mathbb{Z}_{p}$-flow $g$ of $G$ obtained from $g'$ by uncontracting $e_1$ and assigning it flow value $-g(e_2)$ is nowhere-zero, lies in the same connected component of $\mathcal{F}(G,\mathbb{Z}_{p})$ as $f$, and is such that $g(E(G))$ is disjoint from $F$, as desired.
   Hence, we can assume
  without loss of generality that $G$ is 3-edge-connected, and in
  particular, $\e(G)\ge \tfrac32\, \ve(G)$.

As $q$ and $p$ are coprime, $q \pmod p$ has a multiplicative inverse $q^{-1} \pmod p$ in the ring $\mathbb{Z}_p$, and the function $\sigma: z\mapsto q\cdot z$ is an automorphism of $\mathbb{Z}_p$. Note that $\sigma^i$ with $\sigma^i(z)= q^i\cdot z$ is also an automorphism for every $i\in \mathbb{Z}$ (where $q^i\pmod p \equiv  (q^{-1})^{-i}\pmod p$ for $i<0$). 

Pick a random integer $t$ in $[0,8]$, and consider the automorphism $\sigma^t$ of $\mathbb{Z}_p$.
Since $p>q^8\cdot 27+28$, for each element $z\in \mathbb{Z}_p$ there is at most one $i\in [0,8]$ such that $q^i\cdot z\in F \subseteq \mathbb{Z}_p$, and so the expected number of edges $e$ such that $\sigma^t(f(e))\in F$ is at most $|E(G)|/9$.
Thus we may pick $t_0\in [0,8]$ so that for the automorphism $\sigma^{t_0}$ of $\mathbb{Z}_p$, we have that the set $E'$ of edges $e$ with $\sigma^{t_0}(f(e))\in \mathbb{Z}_p\setminus F$ has size at least $8|E(G)|/9$. 
Given that $G$ is $3$-edge-connected, this means $|E'|\geq \frac{4}{3}|V(G)|$.
By Lemma~\ref{lem:kevin}, there is a subgraph $H$ such that $E(H)\subseteq E'$ and $H$ is either a cycle of length at most $5$ or a $(\leq 2)$-subdivision of a $3$-edge-connected graph. 
Contract $H$ into a single vertex $v^*$, and call the resulting graph $G^*$.
Let $f^*$ be the restriction of $f$ to $G^*$.
By induction, there exists a nowhere-zero $\mathbb{Z}_p$-flow $g^*$ in $G^*$ such that $f^*$ and $g^*$ are in the same component of $\mathcal{F}(G^*,\mathbb{Z}_p)$ and $g^*(E(G^*))$ is disjoint from $F$.
Let $(f^*_i)_{0\leq i\leq \ell}$ be a sequence of nowhere-zero $\mathbb{Z}_p$-flows in $G^*$ such that $f^*_0=f^*$, $f^*_{\ell}=g^*$ and for each $1\le i\le \ell$ the flow $f^*_i-f^*_{i-1}$ is supported on a cycle $C^*_i$ of $G^*$.
We now construct a sequence $f_0,f_1,\dots ,f_{\ell}$ of nowhere-zero $\mathbb{Z}_p$-flows in $G$ such that $f_0=f$ and for each $0\le i\le  \ell$:
\begin{enumerate}
    \item\label{item:di1}  the restriction of $f_i$ to $G^*$ is $f^*_i$,
    \item\label{item:di2} if $i\ge 1$, then $f_i$ is in the same component of $\mathcal{F}(G,\mathbb{Z}_p)$ as $f_{i-1}$,
    \item\label{item:di3} there is an integer $t_i$  such that $\sigma^{t_i}(f_i(E(H)))$ is disjoint from $F$.
    \end{enumerate}
    
Since $f_0$ and $t_0$ are already given, we construct the sequence recursively as follows for $i\ge 1$.
Let $C_i$ be a cycle in $G$ whose restriction to $G^*$ is $C^*_i$, and let $\lambda_i\in \mathbb{Z}_p$ be the value so that $f^*_i$ is obtained from $f^*_{i-1}$ by adding $\lambda_i$ along $C^*_i$.
Let $x$ be an arbitrary integer, and consider the set $S_{x}$ of integers $z\in [-27,27]$ such that for some $z'\in [-27,27]$ we have $x+z\equiv q\cdot z' \pmod{p}$.
If $z_1$ and $z_2$ are in $S_x$, then there exist $z'_1, z'_2\in [-27,27]$ and $k\in \mathbb{Z}$ such that $z_1-z_2=q(z'_1-z'_2)+kp$. Since $|z_1-z_2-q(z'_1-z'_2)|\leq 54(q+1)<p$ we must have $k=0$, and hence $z_1-z_2=0=z'_1-z'_2$. 
Hence $|S_x|\leq 1$.
It follows that for each edge $e\in E(C_i)$, there is at most one $z\in F$ such that $\sigma^{t_{i-1}}(f_{i-1}(e)+\lambda_i)+z\in q\cdot F$. Call this element $\phi(e)$ if it exists. 
Likewise, for each $e\in E(H)\setminus E(C_i)$, there is at most one $z\in F$ such that $\sigma^{t_{i-1}}(f_{i-1}(e))+z\in q\cdot F$. Call this element $\phi(e)$ if it exists.
For each $e\in E(H)$ with $\phi(e)$ not yet defined, choose $\phi(e)\in F$ arbitrarily.
Now by Corollary~\ref{cor:28flow}, there is a $28$-flow $h$ of $H$ such that for each edge $e\in E(H)$, $h(e)$ is distinct from the integer from $[-27,27]$ corresponding to $\phi(e)$.

By Lemma~\ref{lem:intflowdec}, there is a sequence  of $28$-flows $h_0,h_1,\dots, h_s$ in $H$ such that $h_0$ is everywhere-zero, $h_s=h$ and for each $j\in \{1,\dots, s\}$ the flow $h_j-h_{j-1}$ is supported on a cycle $D_j$ in $H$ and $h_j(e)$ lies between 0 and $h(e)$ for every edge $e\in E(H)$.
For each $j\in \{0,1,\dots, s\}$, we define a $\mathbb{Z}_p$-flow $f_{i-1,j}$ in $G$ as follows. If $e\in E(G)\setminus E(H)$, let $f_{i-1,j}(e):=f_{i-1}(e)$. If  $e\in E(H)$, let $f_{i-1,j}(e):=f_{i-1}(e)+q^{-t_{i-1}}\cdot h_j(e) \pmod p$.
Observe that \[q^{t_{i-1}}\cdot (f_{i-1}(e)+q^{-t_{i-1}}\cdot h_j(e))\equiv q^{t_{i-1}}\cdot f_{i-1}(e)+h_j(e)\pmod p.\] By \eqref{item:di3},  $q^{t_{i-1}}\cdot f_{i-1}(e) \pmod p \notin F$ and thus $q^{t_{i-1}}\cdot (f_{i-1}(e)+q^{-t_{i-1}}\cdot h_j(e))\not\equiv 0 \pmod p$. This implies that $f_{i-1}(e)+q^{-t_{i-1}}\cdot h_j(e)\not\equiv 0 \pmod p$ and since $f_{i-1}$ is nowhere-zero, $f_{i-1,j}$ is also nowhere-zero (and thus $f_{i-1,s}$ is nowhere-zero), and by construction $f_{i-1,s}$ is in the same component of $\mathcal{F}(G,\mathbb{Z}_p)$ as $f_{i-1}$. 
We now obtain the desired flow $f_i$ from $f_{i-1,s}$ by shifting each edge of $C_i$ by $\lambda_i$. For each $e\in E(H)$, we then have  $q^{t_{i-1}}\cdot f_{i}(e)\pmod p \not\in q\cdot F$ by our choice of $\phi:E(H)\to [-27,27]$. It follows that $q^{t_{i-1}-1}\cdot f_{i}(e)\pmod p \not\in F$ (in particular $f_{i}(e)\not\equiv 0\pmod p$, so $f_i$ is nowhere-zero). By setting $t_i:=t_{i-1}-1$, it follows that we have that $\sigma^{t_i}(f_i(E(H)))$ is disjoint from $F$. This shows that the sequence $f_0,\ldots,f_\ell$ satisfies items \eqref{item:di1}, \eqref{item:di2}, and \eqref{item:di3}, as required.

\smallskip

By definition of $f_\ell$ and $g^*$, we have that for every $e\in E(G)\setminus E(H)$, $f_\ell(e)=g^*(e)\not\in F$. By  \eqref{item:di3}, for every edge $e\in E(H)$, we have $\sigma^{t_\ell}(f_\ell(e))\not\in F$. It now remains to show that there is a flow $g$ in the same component of $\mathcal{F}(G,\mathbb{Z}_p)$ as $f_\ell$ such that $g-f_\ell$ is supported on $E(H)$ and $g(E(H))$ is disjoint from $F$. 
By Euler's Theorem, we have that $q^{\varphi(p)}\equiv 1 \pmod{p}$, and so there is an integer $t'\ge 0$ such that $q^{t_\ell-t'}\equiv 1\pmod{p}$.
By the same argument as above (applied with $\lambda_i=0$, that is without any intermediate shifting along a cycle $C_i$), we can find a sequence of flows $g_0,g_1,\dots, g_{t'}$ such that $g_0=f_\ell$ and for each $i\in \{1,\dots, t'\}$ we have:
\begin{enumerate}
    \item the support of $g_i-g_{i-1}$ is a subset of $E(H)$,
    \item $g_i$ is in the same component of $\mathcal{F}(G,\mathbb{Z}_p)$ as $g_0$, and
    \item for each $e\in E(H)$, we have $\sigma^{t_\ell-i}(g_i(e))\notin F$.
\end{enumerate}
Note that $\sigma^{t_\ell-t'
}$ is the identity by our choice of $t'$.
Now the flow $g_{t'}$ is in the same component of $\mathcal{F}(G,\mathbb{Z}_p)$ as $f$, for each $e\in E(G)\setminus E(H)$ we have $g_{t'}(e)=f_\ell(e)=g^*(e)\notin F$, and for each $e\in E(H)$ we have $g_{t'}(e)=\sigma^{t_\ell-t'}( g_{t'}(e))\notin F$.
This completes the proof.
\end{proof}

In the proof of the next result, we will need to work in a slightly relaxed setting. Consider a graph $G$ and a fixed subset $S\subseteq E(G)$. An $A$-flow is said to be $S$-nowhere-zero if all edges from $S$ are non-zero.  The reconfiguration graph $\mathcal{F}(G,A;S)$ is defined as the graph whose vertices are all $S$-nowhere-zero $A$-flows of $G$, with an edge between two flows if the support of their difference is a cycle, and the reconfiguration graph $\mathcal{F}(G,k;S)$ is defined analogously for integer flows. Note that Observation \ref{obs:intgroup} also applies in the setting of $S$-nowhere-zero flows. We can also restate Lemma \ref{lem:cyclicgroups1} as follows in the context of $S$-nowhere-zero flows (by simply applying it to the support of the flow).

\begin{corollary}
\label{cor:cyclicgroups1}
Let $p$ and $q$ be coprime positive integers with $q\geq 55$ and $p>q^8\cdot 27+28$, and let $G$ be a $2$-edge-connected graph, and let $S$ be a subset of edges of $G$. Let $F:=\{-27,-26,\dots,26,27\}\subseteq \mathbb{Z}_p$.
For any $S$-nowhere-zero $\mathbb{Z}_p$-flow $f$ in $G$, there exists a $\mathbb{Z}_p$-flow $g$ in $G$ such that $\supp(f)=\supp(g)$ (in particular, $g$ is $S$-nowhere-zero), $f$ and $g$ are in the same component of $\mathcal{F}(G,\mathbb{Z}_p;S)$, and $g(S)$ is disjoint from $F$.
\end{corollary}

We will use Corollary \ref{cor:cyclicgroups1} in combination with the next result, which implies that any two  $\mathbb{Z}_p$-flows of a graph $G$ avoiding flow values that are too close to $0\pmod p$ are in the same component of $\mathcal{F}(G,\mathbb{Z}_p)$.

\begin{lemma}\label{lem:cyclicgroups2}
Let $p>36$ be an integer, let $G$ be an oriented $2$-edge-connected graph, let $h'$ be a nowhere-zero $6$-flow in $G$, and let $h_{\mathbf{7}}$ be the nowhere-zero $\mathbb{Z}_p$-flow with $h_{\mathbf{7}}(e)\equiv h'(e) \cdot 7\pmod p$ for each $e\in E(G)$.
If $f$ is a  $\mathbb{Z}_p$-flow in $G$ such that $f(\supp(f))$ is disjoint from $F:=\{-10,-9,\dots,9,10\}\subseteq \mathbb{Z}_p$, then $f$ and $h_{\mathbf{7}}$ are in the same component of $\mathcal{F}(G,\mathbb{Z}_p;\supp(f))$.
\end{lemma}
\begin{proof}
Let $S=\supp(f)$. All flows considered in this proof will be $S$-nowhere-zero (but not necessarily nowhere-zero).
We will first show that the integer counterparts of $f$ and $h_{\mathbf{7}}$ are in the same connected component of $\mathcal{F}(G,p;S)$.

By Theorem \ref{thm:tutteintmod}, there is a $p$-flow $f'$ in $G$ with $f'(e)\equiv f(e) \pmod{p}$ for each $e\in E(G)$. Note that $f$ is an $S$-nowhere-zero $\mathbb{Z}_p$-flow, and thus $f'$ is an $S$-nowhere-zero $p$-flow. We will reconfigure $f'$ into $7h'$ using several intermediate steps, the main of which is another $p$-flow $f_b'+h'$; we now need a few more steps to define $f_b'$.
Let $f_a$ be the $\mathbb{Z}_6$-flow such that $f_a(e)\equiv f'(e)\pmod{6}$ for each $e\in E(G)$, and let $f'_a$ be a $6$-flow such that $f'_a(e)\equiv f_a(e)\pmod{6}$ for each $e\in E(G)$ (the existence of $f'_a$ again follows from Theorem \ref{thm:tutteintmod}).
Since $f(S)$ is disjoint from $F$, we have that $f'$ is an $S$-nowhere-zero $(p-10)$-flow and so $f'_b:=f'-f'_a$ is an $S$-nowhere-zero $(p-5)$-flow. Observe that for every edge $e\notin S$, $f'(e)=0$, $f_a(e)\equiv 0 \pmod 6$, $f_a'(e)=0$, and thus $f_b'(e)=0$.
Note that for each $e\in E(G)$, $f'_b(e)\equiv 0\pmod{6}$. For the convenience of the reader, the relations between the different flows introduced above are illustrated in Figure \ref{fig:diag}.

\begin{figure}[htb]
    \centering
    \includegraphics[scale=1.2]{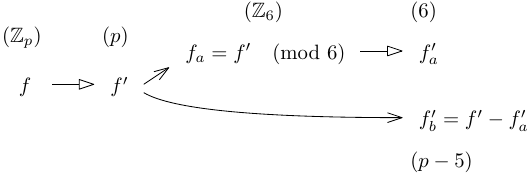}
    \caption{The flows originating from $f$.}
    \label{fig:diag}
\end{figure}

By Lemma~\ref{lem:intflowdec} applied to the flow $h'-f'_a$, there is a sequence of flows $h'_0,h'_1,\dots, h'_s$ such that $h'_0=f'_a$ and $h'_s=h'$ and for each $i\in \{1,\dots, s\}$:
\begin{enumerate}
    \item the support of $h'_i-h'_{i-1}$ is a cycle,
    \item for each $e\in E(G)$, $h'_i(e)$ lies between $f'_a(e)$ and $h'(e)$.
    \end{enumerate}
    
    For each $i\in \{0,1,\dots, s\}$, let $f'_{i}=f'_b+h'_{i}$, so in particular $f'_0=f'$ and $f'_s=f'_b+h'$. Since $f'_b$ is a $(p-5)$-flow and $h'$ is a $6$-flow, $f'_b+h'$ is a $p$-flow.
    By construction, for each $i\in \{0,1,\dots, s\}$ and each $e\in E(G)$, we have that $f'_i(e)$ is between $f'(e)$ and $f'(e)-f'_a(e)+h'(e)$.
    This means $|f'(e)-f'_i(e)|\leq |h'(e)-f'_a(e)|\leq 10$.
    By definition of $f$ and $f'$, we have $10<|f'(e)|<p-10$ for any edge $e\in S$, and thus  $f'_i$ is an $S$-nowhere-zero $p$-flow. We deduce that $f'$ and $f'_s=f_b'+h'$ are in the same component of $\mathcal{F}(G,p;S)$.

    Since $f'_b(e)\equiv 0\pmod{6}$ for each $e\in E(G)$, we have that $f^*:=\frac{1}{6}f'_b$ is a $\lfloor\frac{p}{6}\rfloor$-flow. 
    By Lemma~\ref{lem:intflowdec} applied to the flow $h'-f^*$, there is a sequence of flows $h^*_0,h^*_1,\dots, h^*_t$ such that $h^*_0=f^*$ and $h^*_t=h'$ and for each $i\in \{1,\dots, t\}$:
\begin{enumerate}
    \item the support of $h^*_i-h^*_{i-1}$ is a cycle,
    \item for each $e\in E(G)$, $h^*_i(e)$ lies between $f^*(e)$ and $h'(e)$.
    \end{enumerate}
    
    For each $i\in \{0,\dots, t\}$, let $f^*_i:=6h^*_{i}+h'$, so that in particular $f^*_0=f_b'+h'=f'_s$ and $f^*_t=7h'$.
    By construction, for each $i\in \{1,\dots ,t\}$ and each $e\in E(G)$ we have that $f^*_i(e)$ is between $f'_s(e)$ and $7h'(e)$ and that $f^*_i(e)\equiv h'(e)\pmod{6}$, so $f^*_i$ is a nowhere-zero $p$-flow.
    Hence $f'_s$ is in the same component of $\mathcal{F}(G,p)$ as $7h'$, and thus also in the same component of $\mathcal{F}(G,p;S)$

    We now have that $f'$ is in the same component of $\mathcal{F}(G,p;S)$ as $7h'$.
    By Observation \ref{obs:intgroup} (applied in the context of $S$-nowhere-zero-flows), it follows that $f$ is in the same component of $\mathcal{F}(G,\mathbb{Z}_p;S)$ as $h_{\mathbf{7}}$, as desired.
\end{proof}
As a corollary of Lemmas~\ref{lem:cyclicgroups1} and~\ref{lem:cyclicgroups2}, we obtain the following.

\begin{corollary}\label{cor:largecyclicpq}
If $p$ and $q$ are coprime positive integers with $q\geq 55$ and $p>q^8\cdot 27+28$, $A$ is an abelian group and $G$ is a 2-edge-connected graph, then $\mathcal{F}(G,\mathbb{Z}_p\times A)$ is connected.
\end{corollary}
\begin{proof}
 Let $f=(f_1,f_2)$ be a nowhere-zero $(\mathbb{Z}_p\times A)$-flow in $G$. Let $S$ be the support of $f_1$.
 By Corollary~\ref{cor:cyclicgroups1} there is an $S$-nowhere-zero $\mathbb{Z}_p$-flow $f_1'$ of $G$ such that $f_1$ and $f_1'$ are in the same component of $\mathcal{F}(G,\mathbb{Z}_p;S)$, and $f_1'(S)$ is disjoint from $\{-27,-26,\dots,26,27\}$. By Lemma~\ref{lem:cyclicgroups2}, $f_1'$ is in the same connected component of $\mathcal{F}(G,\mathbb{Z}_p;S)$ as the nowhere-zero flow $h_{\mathbf{7}}$ (defined in the statement of Lemma~\ref{lem:cyclicgroups2}).

 Note that any path between $f_1$ and $h_{\mathbf{7}}$ in  $\mathcal{F}(G,\mathbb{Z}_p;S)$ lifts to a path between $f=(f_1,f_2)$ and $(h_{\mathbf{7}},f_2)$ in $\mathcal{F}(G,\mathbb{Z}_p\times A)$, since $f$ is nowhere-zero and thus for every edge $e\notin S$, $f_2(e)\ne 0_A$. 
 This shows that $f$ is in the same component of $\mathcal{F}(G,\mathbb{Z}_p\times A)$ as $(h_{\mathbf{7}},f_2)$. 
 
 Similarly, any nowhere-zero $(\mathbb{Z}_p\times A)$-flow $g=(g_1,g_2)$ is in the same connected component of $\mathcal{F}(G,\mathbb{Z}_p\times A)$ as $(h_{\mathbf{7}},g_2)$. Since $h_{\mathbf{7}}$ is a nowhere-zero flow, it follows from Observation \ref{obs:sum2} that $f$ and $g$ are in the same connected component of $\mathcal{F}(G,\mathbb{Z}_p\times A)$. This completes the proof.
\end{proof}

We directly obtain the following.

\begin{corollary}\label{cor:largecyclic}
For any prime power $p> 56^8\cdot 27+28$, any abelian group $A$, and any $2$-edge-connected graph $G$, $\mathcal{F}(G,\mathbb{Z}_p\times A)$ is connected.
\end{corollary}

\begin{proof}
Take $q\in \{55,56\}$ coprime with $p$ and apply Corollary \ref{cor:largecyclicpq}.
\end{proof}

\subsection{Large abelian groups}\label{sec:largeabelian}

We now revisit the components of the proof of Theorem \ref{thm:groupflow}, which states that for any 2-edge-connected graph $G$,
$\mathcal{F}(G,\mathbb{Z}_2^{8})$ is connected. The size of the group in the proof was optimized and for this it was convenient to view $\mathbb{Z}_2^{8}$ as a vector space over $\textrm{GF}(2)$. Here we go back to the group view. In this context, we say that for two abelian groups $A$ and $B$, an $(A\times B)$-flow $f$ in a graph $G$ is \emph{composite} if $f=(f_A,f_B)$, where $f_A$ is a nowhere-zero $A$-flow in $G$ and $f_B$ is a nowhere-zero $B$-flow in $G$.

\medskip

The proof of Lemma \ref{lem:composite} readily extends to the following slightly more general result.

\begin{lemma}\label{lem:composite2}
Let $G$ be a 2-edge-connected graph and let $\Gamma_1,\ldots,\Gamma_k$ be $k\ge 8$ non-trivial abelian groups. Let $\Gamma=\Gamma_1\times \cdots \times \Gamma_k$. Then all composite $\Gamma$-flows in $G$ are in the same connected component of $\mathcal{F}(G,\Gamma)$.
\end{lemma}

We now explain how to modify the proof of Lemma \ref{lem:gentocompo} to prove the following result (note the slightly weaker bound on the number of groups).

\begin{lemma}\label{lem:gentocompo2}
Let $G$ be a 2-edge-connected graph and let $\Gamma_1,\ldots,\Gamma_k$ be $k\ge 46$ non-trivial abelian groups. Let $\Gamma=\Gamma_1\times \cdots \times \Gamma_k$. For any nowhere-zero $\Gamma$-flow $f$ in $G$, there exists a composite nowhere-zero $\Gamma$-flow $g$ in $G$ such that $f$ and $g$ are in the same connected component of $\mathcal{F}(G,\Gamma)$.
\end{lemma}

\begin{proof}
The proof proceeds almost exactly as the proof of Lemma \ref{lem:gentocompo}. The only two differences are the following. 
First, in order to define a random splitting $A\oplus B$ of $\Gamma$, we take a subset $I\subseteq [k]$ with $|I|=k-5$ uniformly at random and we set $A:=A_1\times \cdots\times A_k$, where $A_i=\Gamma_i$ if $i\in I$ and $A_i=\{0_{\Gamma_i}\}$ otherwise. Similarly, $B:=B_1\times \cdots\times  B_k$, where $B_i=\{0_{\Gamma_i}\}$ if $i\in I$ and $B_i=\Gamma_i$ otherwise. Note that $A$ and $B$ each contain at least 6 elements, and every element $x$ of $\Gamma$ can be written uniquely as $x=x_A+x_B$, with $x_A\in A$ and $x_B\in B$, and similarly every $\Gamma$-flow $f$ can be written uniquely as  $f=f_A+f_B$, where $f_A$ is an $A$-flow and $f_B$ is a $B$-flow.

For every non-zero element $x\in \Gamma$, the probability that $x_A$ is non-zero is at least $\tfrac{k-5}{k}\ge  41/46$. It follows that in expectation, for every nowhere-zero $\Gamma$-flow $f$ in $G$, the number of edges $e$ of $G$ for which $f_A(e)$ is non-zero is at least \[\tfrac{41}{46}\cdot \e(G)\ge \tfrac{41}{46}\cdot \tfrac32\, \ve(G)> \tfrac43\, \ve(G).\] From now on, we fix a pair of complementary subgroups $A,B$ as above, each being the direct product of at least 5 non-trivial abelian groups, and such that the number of edges $e$ of $G$ for which $f_A(e)$ is non-zero is greater than $\tfrac43\, \ve(G)$.

The proof now follows the lines of the proof of Lemma \ref{lem:gentocompo}, with only the following slight modifications. We use
\begin{itemize}
    \item Observation \ref{obs:sum2} instead of Observation \ref{obs:sum},
    \item Lemma \ref{lem:composite2} instead of Lemma \ref{lem:composite}, and \item  Corollary \ref{cor:gco} instead of Corollary \ref{cor:raphael}.
\end{itemize} 
As each of $A$ and $B$ is the direct product of at least 5 non-trivial abelian groups, each of $A$ and $B$ is the direct product of two abelian groups, one with at least 7 elements and the other with at least 4 elements, so Corollary \ref{cor:gco} can indeed be applied to $A$ and $B$.
\end{proof}

We can now prove the main result of this section.

\begin{theorem}\label{thm:largegroup}
For every 2-edge-connected graph $G$ and every abelian group $A$ with $|A|\ge (56^8\cdot 27+28)^{45}+1$, $\mathcal{F}(G,A)$ is connected.
\end{theorem}

\begin{proof}
Write $A=\mathbb{Z}_{i_1}\times \cdots \times \mathbb{Z}_{i_k}$, where each $i_j$ is a prime power. Let $f$ and $g$ be two nowhere-zero $A$-flows in $G$.

Assume first that $k\ge 46$. Then by Lemma \ref{lem:gentocompo2} there are composite $A$-flows $f^*$ and $g^*$ such that $f$ and $f^*$ are in the same component of $\mathcal{F}(G,A)$, and $g$ and $g^*$ are in the same component of $\mathcal{F}(G,A)$. By Lemma \ref{lem:composite2}, $f^*$ and $g^*$ are in the same component of $\mathcal{F}(G,A)$, which implies that $f$ and $g$ are in the same component of $\mathcal{F}(G,A)$.

Assume now that $k\le 45$. As $|\Gamma|\ge (56^8\cdot 27+28)^{45}+1$, there is  an integer $1\le j \le k$ such that $i_j>56^8\cdot 27+28$, and it then follows from Corollary \ref{cor:largecyclic} that $f$ and $g$ are in the same connected component of $\mathcal{F}(G,\Gamma)$.
\end{proof}

\section{Integer flow reconfiguration in general graphs}\label{sec:intflows}

\begin{theorem}\label{thm:integerflow}
For every 2-edge-connected graph $G$, there exists an integer $k\ge 6$ such that $\mathcal{F}(G,k)$ is connected.
\end{theorem}

\begin{proof}
Let $t=5\,\e(G)+1$ and $k= 36\,\e(G) (\ve(G)+1)\ge (2t+1)(\ve(G)+1)+2$.

\smallskip

We first fix some target nowhere-zero $k$-flow $g$ in $G$, and show that all nowhere-zero $k$-flows  in $G$ are in the same connected component as $g$ in $\mathcal{F}(G,k)$. To define $g$, we take a nowhere-zero 6-flow $h$ of $G$ (whose existence follows from \cite{Seymour6flows}), and define $g(e)=(\tfrac{k}{6}+1)\cdot h(e)$ for any edge $e$ in $G$. As $k$ is divisible by 36, it follows from the definition of $h$ that $g(e)\not\equiv 0 \pmod 6$ for every edge $e$ in $G$.

Consider any nowhere-zero $k$-flow $f$ in $G$. The first step is to modify $f$ to ensure that all flow values of $f$ in $G$ are sufficiently far away from 0 and $\pm(k-1)$. 

As long as there exists an edge $e$ such that $|f(e)|<t$, or $|f(e)|>k-t$, we do the following. Up to reversing some edges (and replacing $f(e)$ by $-f(e)$ on each of these edges) we can assume that  all flow values are positive. We consider a directed cycle $C$ containing $e$ (such a cycle necessarily exists). We observe that there is an integer $t\le a\le k-t$ such that $C$ does not have any flow value in the interval $[a-t,a+t]$, since otherwise at least $\tfrac{k-2}{2t+1}>|V(G)|$ different flow values appear on $C$, which is a contradiction. We then subtract flow value $a$ along $C$, and in the resulting flow, all flow values of $C$ are at least $t$ and at most $k-t$ in absolute value, as desired.

We can now assume that for any edge $e$ in $G$, $t\le |f(e)|\le k-t$.

\smallskip

Next, we write every element $a\in [-k+1, k-1]$ as $a=6 a_1  +a_2$, where $0\le a_2<6$. By extension, we write $f=6 f_1+f_2$ and $g=6 g_1+ g_2$, where $f_2$ and $g_2$ have values in $[0,5]$.

By Lemma \ref{lem:linear}, we can decompose the $\mathbb{Z}_6$-flow $(g-f) \pmod 6\equiv (g_2-f_2)\pmod 6$ of $G$ into a sum of at most $\e(G)$ $\mathbb{Z}_6$-flows whose support are cycles of $G$. It follows that there is a sequence of $r\le \e(G)$ cycles $(C_i)_{1\le i \le r}$  of $G$ and a sequence of values $(\lambda_i)_{1\le i \le r}$ from $[1,5]$ such that $g_2 \pmod 6$ can be obtained from $f_2 \pmod 6$ by adding  $\lambda_i \pmod 6$ along $C_i$ for any $i=1,\ldots,r$. Consider the integer flow $f'$ of $G$ obtained from $f$ by adding  $\lambda_i$ along $C_i$ for each $i=1,\ldots,r$. Observe that $f'(e)\equiv g(e) \pmod 6$ for any edge $e$ of $G$, and recall that by the definition of $g$, $g(e) \not\equiv 0 \pmod 6$ for every edge $e$ of $G$. Each time we add some $\lambda_i$ along some cycle $C_i$, the flow values on the edges of $C$ change by at most $5$. As $r\le \e(G)$, all flow values remain in $[t-5\e(G),k-t+5\e(G)]$ in absolute value during the whole procedure. As $t=5\e(G)+1$, all intermediate flows are nowhere-zero $k$-flows. 

\smallskip

We now consider the $k$-flow $g-f'$ in $G$. By the paragraph above, all flow values of $g-f'$ are divisible by $6$. By considering the $(k/6)$-flow $(g-f')/6$ and applying Lemma \ref{lem:intflowdec}, it follows that there is a sequence of cycles $(D_i)_{1\le i \le s}$  of $G$ and a sequence of values $(\kappa_i)_{1\le i \le s}$ from $[-k/6,k/6]$ such that 
\begin{enumerate}
    \item $g$ can be obtained from $f'$ by adding  $6\kappa_i$ along $D_i$ for each $i=1,\ldots,s$, and
    \item for each edge $e$, all flow values of $e$ remain between $f'(e)$ and $g(e)$, and in particular in $[-k+1,k-1]$, during the procedure.
\end{enumerate} 
As all flow values of $f'$ are distinct from 0 modulo $6$, all flow values remain distinct from 0 modulo $6$, and in particular non-zero, during this procedure. It follows that all intermediate flows are nowhere-zero $k$-flows, as desired.
\end{proof}

\section{Conclusion}\label{sec:ccl}

\subsection*{Flows with boundaries} A natural generalization of $A$-flows is that of \emph{$A$-flows with boundary} (see \cite{BoundaryThesis} for a survey). We start with a \emph{boundary} $\beta$, which is a function $\beta:V(G)\to A$ such that $\sum_{v\in V(G)} \beta(v)=0_A$, and the goal is to find a function $f:E(G)\to A$ such that  $\delta f (v) = \beta(v)$ for any vertex $v$ of $G$ (see Section \ref{sec:def} for the definition of $\delta f$). By taking $\beta=0_A$, we recover the notion of an $A$-flow. We can ask more generally if for all boundaries $\beta$ of a graph $G$, all the nowhere-zero $A$-flows with boundary $\beta$ in $G$ are connected. We have chosen to state our results only in the usual setting of nowhere-zero flows without boundary to avoid unnecessary technicalities, but it turns out that the results in Sections \ref{sec:duality} and \ref{sec:ab} hold in the more general context of flows with boundaries. For Section \ref{sec:duality} we just need a version of Theorem \ref{coloring_degeneracy} which holds for \emph{group coloring} (the dual version of flows with boundaries) instead of classical coloring, but the original proof of Theorem \ref{coloring_degeneracy} readily applies to this setting as well. 

\subsection*{Final remarks} The main questions left open by our work are Problem \ref{conj:iflow2}, and whether the cardinality of the group $\mathbb{Z}_2^8$ in our positive answer to Problem \ref{conj:gflow2} can be improved.

\subsection*{Acknowledgements} We thank Zolt\'an Szigeti and Clément Legrand-Duchesne for helpful discussions on frozen configurations in nowhere-zero $\mathbb{Z}_4$-flows and $\mathbb{Z}_5$-flows, and Dan Cranston for pointing out errors in the original proofs of Lemma~\ref{lem:cyclicgroups1} and  Theorem \ref{thm:largegroup}. We would also like to thank Patryk Morawski and Yuval Wigderson for their kind help with finding and verifying the example from Figure~\ref{fig:Z5frozen}. The results from Sections \ref{sec:ab} and \ref{sec:intflows} were obtained during the 2026 Barbados Graph Theory Workshop held at the Bellairs Research Institute in January 2026. Thanks to the organizers and the other workshop participants for creating a productive working atmosphere.

The last author gratefully acknowledges funding from the SNSF Ambizione Grant No. 216071.

\bibliographystyle{plain}
\bibliography{references.bib}

\needspace{8\baselineskip} 

\end{document}